\documentclass[10pt]{amsart}

\usepackage{charter}

\usepackage{setspace}
\usepackage{macros}
\standardsettings
\colorcommentstrue

\newcommand{\mult}{\times}

\newcommand{\Hyp}{\scrH}

\newcommand{\borel}{E}

\newcommand{\norm}{s}

\newcommand{\qdim}{N}
\newcommand{\pdim}{M}
\newcommand{\bothdim}{{M,N}}
\newcommand{\Index}{n}

\draftfalse

\theoremstyle{theorem}
\maketheorem{reduction}{Reduction}

\begin{document}
\title[Extremality and dynamically defined measures, I]{Extremality and dynamically defined measures, part I: Diophantine properties of quasi-decaying measures}

\authortushar\authorlior\authordavid\authormariusz

%\date{}
%\dedicatory{}

\begin{Abstract}
We present a new method of proving the Diophantine extremality of various dynamically defined measures, vastly expanding the class of measures known to be extremal. This generalizes and improves the celebrated theorem of Kleinbock and Margulis ('98) resolving Sprind\v zuk's conjecture, as well as its extension by Kleinbock, Lindenstrauss, and Weiss ('04), hereafter abbreviated KLW. As applications we prove the extremality of all hyperbolic measures of smooth dynamical systems with sufficiently large Hausdorff dimension, of the Patterson--Sullivan measures of all nonplanar geometrically finite groups, and of the Gibbs measures (including conformal measures) of infinite iterated function systems. The key technical idea, which has led to a plethora of new applications, is a significant weakening of KLW's sufficient conditions for extremality.

In Part I, we introduce and develop a systematic account of two classes of measures, which we call {\it quasi-decaying} and {\it weakly quasi-decaying}. We prove that weak quasi-decay implies strong extremality in the matrix approximation framework, thus proving a conjecture of KLW. We also prove the ``inherited exponent of irrationality'' version of this theorem, describing the relationship between the Diophantine properties of certain subspaces of the space of matrices and measures supported on these subspaces.

In subsequent papers, we exhibit numerous examples of quasi-decaying measures, in support of the thesis that ``almost any measure from dynamics and/or fractal geometry is quasi-decaying''. 
We also discuss examples of non-extremal measures coming from dynamics, illustrating where the theory must halt.
\end{Abstract}
\maketitle

\section{Introduction}

In this series of papers we address a central problem in the flourishing area of metric Diophantine approximation on manifolds and measures: an attempt to exhibit a possibly widest natural class of sets and measures for which most points are not very well approximable by ones with rational coordinates.

Fix $d\in\N$. The quality of rational approximations to a vector $\xx\in\R^d$ can be measured by its \emph{exponent of irrationality}, which is defined by the formula
\[
\omega(\xx) = \limsup_{\pp/q\in\Q^d} \frac{-\log\|\xx - \pp/q\|}{\log(q)},
\]
where the limsup is taken over any enumeration of $\Q^d$, and $\|\cdot\|$ is any norm on $\R^d$. Another interesting quantity is the \emph{exponent of multiplicative irrationality}, which is the number
\[
\omega_\mult(\xx) = \limsup_{\pp/q\in\Q^d} \frac{-\log\prod_{i = 1}^d |x_i - p_i/q|}{\log(q)}\cdot
\]
It follows from a pigeonhole argument that $\omega(\xx) \geq 1 + 1/d$ and $\omega_\mult(\xx) \geq d + 1$. A vector $\xx$ is said to be \emph{very well approximable} if $\omega(\xx) > 1 + 1/d$, and \emph{very well multiplicatively approximable} if $\omega_\mult(\xx) > d + 1$. We will denote the set of very well (multiplicatively) approximable vectors by $\mathrm{VW(M)A}_d$. It is well-known that $\VWA_d$ and $\VWMA_d$ are both Lebesgue nullsets of full Hausdorff dimension, and that $\VWA_d \subset \VWMA_d$.

A measure $\mu$ on $\R^d$ is \emph{extremal} if $\mu(\VWA_d) = 0$, and \emph{strongly extremal} if $\mu(\VWMA_d) = 0$. Extremality was first defined by V.~G.~Sprind\v zuk, who conjectured that the Lebesgue measure of any nondegenerate manifold is extremal. This conjecture was proven by D.~Y.~Kleinbock and G.~A.~Margulis \cite{KleinbockMargulis}, and later strengthened by D.~Y.~Kleinbock, E.~Lindenstrauss, and B.~Weiss (hereafter abbreviated ``KLW'') in \cite{KLW}, who considered a class of measures which they called ``friendly'' and showed that these measures are strongly extremal. However, their definition is somewhat rigid and many interesting measures, in particular ones coming from dynamics, do not satisfy their condition. In this paper, we study a much larger class of measures, which we call \emph{weakly quasi-decaying}, such that every weakly quasi-decaying measure is strongly extremal (Corollary \ref{corollaryKLW}). This class includes a subclass of \emph{quasi-decaying} measures, which are the analogue of KLW's ``absolutely friendly'' measures.\Footnote{The terminology ``absolutely friendly'' was not used by KLW and first appeared in \cite{PollingtonVelani}; however, several theorems about absolute friendliness had already appeared in \cite{KLW} without using the terminology.}

In the current paper (Part I), we demonstrate the most basic properties of the quasi-decay condition, including the facts that every exact dimensional measure is quasi-decaying, and that every quasi-decaying measure is extremal, which we prove using an elementary argument. We also prove the result stated above that every weakly quasi-decaying measure is strongly extremal (in particular verifying a conjecture of KLW), as well as considering the approximation properties of quasi-decaying measures on the space of matrices and on affine subspaces of $\R^d$. In particular we generalize results of some recent papers regarding approximation of friendly measures in the matrix framework \cite{KMW, ABRdS} (cf. Theorems \ref{theoremKLW} and \ref{theoremexponents} below).

\tableofcontents

{\bf Notation.} For the reader's convenience we summarize a list of notations and conventions:

\begin{convention}
The symbols $\lesssim$, $\gtrsim$, and $\asymp$ will denote coarse asymptotics; a subscript of $\plus$ indicates that the asymptotic is additive, and a subscript of $\times$ indicates that it is multiplicative. For example, $A\lesssim_\times B$ means that there exists a constant $C > 0$ (the \emph{implied constant}) such that $A\leq C B$.

If $\mu$ and $\nu$ are measures, then $\nu\lesssim_\times\mu$ means that there exists a constant $C > 0$ such that $\nu\leq C\mu$.
\end{convention}

\begin{convention}
In this paper, all measures and sets are assumed to be Borel, and measures are assumed to be locally finite.
\end{convention}

\begin{convention}
The symbol $\triangleleft$ will be used to indicate the end of a nested proof.
\end{convention}

\begin{figure}[h!]
\begin{tabular}{ll}
\hline 
$\omega(\xx)$ & The \emph{exponent of irrationality} of $\xx\in\R^d$ defined as\\
& \hspace{0.1 in} $\omega(\xx) \df \limsup \left\{ \frac{-\log\|\xx - \pp/q\|}{\log(q)} : \pp/q\in\Q^d \right\}$\\
$\omega_\mult(\xx)$ & The \emph{exponent of multiplicative irrationality} of $\xx\in\R^d$ defined as\\
& \hspace{0.1 in} $\omega_\mult(\xx) \df \limsup \left\{ \frac{-\log\prod_{i = 1}^d |x_i - p_i/q|}{\log(q)} : \pp/q\in\Q^d \right\}$\\
$\mathrm{VWA}_d$ & The set of very well approximable vectors in $\R^d$\\
$\mathrm{VWMA}_d$ & The set of very well multiplicatively approximable vectors in $\R^d$\\
$\thickvar S\rho$ & The closed $\rho$-thickening of $S\subset\R^d$~:~ $\thickvar S\rho \df \{x\in\R^d:\dist(x,S) \leq \rho\}$\\
$\thickopenvar S\rho$ & The open $\rho$-thickening of $S\subset\R^d$~:~ $\thickopenvar{S}{\rho} \df \{x\in\R^d:\dist(x,S) < \rho\}$\\
$A\wedge B$ & The minimum of $A$ and $B$\\
$A\vee B$ & The maximum of $A$ and $B$\\
$\Hyp$ & The collection of affine hyperplanes in $\R^d$\\
$\Supp(\mu)$ & The topological support of a measure $\mu$\\
$B(\xx,\rho)$ & The closed ball centered at $\xx \in \R^d$ of radius $\rho>0$\\
$\dist(\yy,\LL)$ & $\dist(\yy,\LL) \df \inf \{ d(\yy,\xx) : \xx \in \LL \}$\\
$\|d_\LL\|_{\mu,B}$ & For a hyperplane $\LL \in \HH$ and a ball $B$ centered at $\Supp(\mu)$\\ 
& \hspace{0.1 in} $\|d_\LL\|_{\mu,B} \df \sup\{\dist(\yy,\LL) : \yy\in B\cap\Supp(\mu)\}$\\
$\MM_\bothdim$ & The set of $M \times N$ matrices with real entries\\
$\omega(\bfA)$ & The \emph{exponent of irrationality} of $\bfA\in\MM$ defined as\\
& \hspace{0.1 in} $\omega(\bfA) \df \limsup \left\{ \frac{-\log\|\bfA\qq - \pp\|}{\log\|\qq\|} : \qq\in\Z^\qdim\butnot\{\0\} , \pp\in\Z^\pdim \right\}$\\
$\omega_\mult(\bfA)$ & The \emph{exponent of multiplicative irrationality} of $\bfA\in\MM$ defined as\\
& \hspace{0.1 in} $\omega_\mult(\bfA) \df \limsup \left\{ \frac{-\log\prod_{i = 1}^\pdim |(\bfA\qq - \pp)_i|}{\log \prod_{j = 1}^\qdim |q_j|\vee 1} : \qq\in\Z^\qdim\butnot\{\0\} , \pp\in\Z^\pdim \right\}$\\
$\mathrm{VWA}_{\bothdim}$ & The set of very well approximable $\pdim\times \qdim$ matrices in $\MM$\\
$\mathrm{VWMA}_{\bothdim}$ & The set of very well multiplicatively approximable  $\pdim\times \qdim$ matrices in $\MM$\\
$\|f\|_B$ & $\|f\|_B \df \sup \{ \|f(\xx)\| : \xx\in B \}$\\
$\|f\|_{\CC^\epsilon,B}$ & $\|f\|_{\CC^\epsilon} \df \sup \left\{ \frac{\|f(\yy) - f(\xx)\|}{\|\yy - \xx\|^\epsilon} : \xx,\yy\in B \right\}$ for $f:B\to\R$ a function of class $\CC^{\ell + \epsilon}$\\
$\Delta,\|f\|,\|f\|_{\CC^\epsilon}$ & $\Delta \df B(\0,1)$, $\|f\| \df \|f\|_\Delta$, $\|f\|_{\CC^\epsilon} \df \|f\|_{\CC^\epsilon,\Delta}$\\
$\Lambda_0$ & $\Lambda_0 \df \Z^{\pdim + \qdim} \subseteq \R^{\pdim + \qdim}$\\
$u_\bfA, g_\tt$ & See Section \ref{subsectioncorrespondence}\\
$\mfa$ & $\mfa \df \{\tt\in\R^{\pdim + \qdim} : \sum t_i = 0\}$\\
$\mfa_+, \mfa_+ ^*$ & See Section \ref{subsectioncorrespondence}\\
$\omega(\bfA;\SS,\norm)$ & See Section \ref{subsectioncorrespondence}\\
$\Delta(\Lambda)$ & Given a lattice $\Lambda \subseteq \R^{\pdim + \qdim}$, $\Delta(\Lambda) \df -\log\min\big\{\|\rr\| : \rr\in\Lambda\butnot\{\0\}\big\}$\\
\hline 
\end{tabular}
\end{figure}

\vskip13pt
{\bf Acknowledgements.} The authors thank Barak Weiss for helpful comments. The first-named author was supported in part by a 2014-2015 Faculty Research Grant from the University of Wisconsin--La Crosse. The second-named author was supported in part by the Simons Foundation grant \#245708. The third-named author was supported in part by the EPSRC Programme Grant EP/J018260/1. The fourth-named author was supported in part by the NSF grant DMS-1361677. The authors thank an anonynmous referee for helpful comments.

\subsection{Four conditions which imply strong extremality}
\label{subsectionfourconditions}
We begin by recalling the definitions of friendly and absolutely friendly measures, in order to compare these definitions with our new definitions of quasi-decaying and weakly quasi-decaying measures. The definitions given below are easily seen to be equivalent to KLW's original definitions in \cite{KLW}.
\begin{definition}
\label{definitionabsolutelyfriendly}
Let $\mu$ be a measure on an open set $U\subset\R^d$, and let $\Supp(\mu)$ denote the topological support of $\mu$.
\begin{itemize}
\item $\mu$ is called \emph{absolutely decaying (resp. decaying)} if there exist $C_1,\alpha > 0$ such that for all $\xx\in\Supp(\mu)$, $0 < \rho \leq 1$, $\beta > 0$, and $\LL\in\Hyp$, if $B = B(\xx,\rho) \subset U$ then
\begin{align}
\label{absolutelydecaying}
\mu\big(\thickopenvar\LL{\beta\rho}\cap B\big) &\leq C_1 \beta^\alpha\mu(B) & \text{(absolutely decaying)}
\end{align}
or
\begin{align}
\label{decaying}
\mu\big(\thickopenvar\LL{\beta \|d_\LL\|_{\mu,B}}\cap B\big) &\leq C_1 \beta^\alpha\mu(B) & \text{(decaying)},
\end{align}
respectively, where
\[
\|d_\LL\|_{\mu,B} := \sup\{\dist(\yy,\LL):\yy\in B\cap\Supp(\mu)\}.
\]
\item $\mu$ is called \emph{nonplanar} if $\mu(\LL) = 0$ for all $\LL\in\Hyp$. Note that every absolutely decaying measure is nonplanar. Moreover, the decaying and nonplanarity conditions can be combined notationally by using closed thickenings rather than open ones: a measure $\mu$ is decaying and nonplanar if and only if there exist $C_1,\alpha > 0$ such that for all $\xx\in\Supp(\mu)$, $0 < \rho \leq 1$, $\beta > 0$, and $\LL\in\Hyp$, if $B = B(\xx,\rho) \subset U$ then
\begin{align}
\label{decayingprime}
\mu\big(\thickvar\LL{\beta \|d_\LL\|_{\mu,B}}\cap B\big) &\leq C_1 \beta^\alpha\mu(B). & \text{(decaying and nonplanar)}
\end{align}
\item $\mu$ is called \emph{Federer} (or \emph{doubling}) if for some (equiv. for all) $K > 1$, there exists $C_2 > 0$ such that for all $\xx\in\Supp(\mu)$ and $0 < \rho\leq 1$, if $B(\xx,K\rho)\subset U$ then
\begin{equation}
\label{federer}
\mu\big(B(\xx,K\rho)\big) \leq C_2\mu\big(B(\xx,\rho)\big).
\end{equation}
\end{itemize}
If $\mu$ is Federer, decaying, and nonplanar, then $\mu$ is called \emph{friendly}; if $\mu$ is both absolutely decaying and Federer, then $\mu$ is called \emph{absolutely friendly}.\Footnote{As KLW put it, the word ``friendly'' is ``a somewhat fuzzy abbreviation of \emph{Federer, nonplanar, and decaying}''.} When the open set $U$ is not explicitly mentioned, we assume that it is all of $\R^d$; otherwise we say that $\mu$ is absolutely decaying, friendly, etc. ``relative to $U$''.
\end{definition}
The main relations between friendly and absolutely friendly measures are as follows:
\begin{itemize}
\item[(i)] every absolutely friendly measure is friendly;
\item[(ii)] the Lebesgue measure of a nondegenerate submanifold of $\R^d$ (see Definition \ref{definitionnondegenerate} for the definition) is friendly but not absolutely friendly;
\item[(iii)] \cite[Theorem 2.1]{KLW} more generally, the image of an absolutely friendly measure under a nondegenerate embedding is friendly.
\end{itemize}
The main result of \cite{KLW} states that every friendly measure is strongly extremal; together with (ii), this provides a proof of Sprind\v zuk's conjecture.

The distinction between friendly and absolutely friendly measures is a fundamental part of the theory; for example, (iii) would be false if we replaced the hypothesis ``absolutely friendly'' by ``friendly''. So any good generalization of friendliness should also respect the ``friendliness-type condition/absolute friendliness-type condition'' distinction. Thus we will define two versions of the quasi-decay condition, one to correspond with friendliness and the other to correspond with absolute friendliness. Since, in our experience, the ``absolute'' versions of these conditions are more fundamental than the ``non-absolute'' versions, we call our condition which corresponds to absolute friendliness the ``quasi-decay'' condition and we call our condition which corresponds to friendliness the ``weak quasi-decay'' condition.

\begin{definition}
\label{definitionQD}
Let $\mu$ be a measure on $\R^d$ and consider $\xx\in \borel\subset\R^d$. We will say that $\mu$ is \emph{quasi-decaying (resp. weakly quasi-decaying) at $\xx$ relative to $\borel$} if for all $\gamma > 0$, there exist $C_1,\alpha > 0$ such that for all $0 < \rho \leq 1$, $0 < \beta \leq \rho^\gamma$, and $\LL\in\Hyp$, if $B = B(\xx,\rho)$ then
\begin{align}
\label{QDwithE}
\mu\left(\thickvar{\LL}{\beta\rho}\cap B\cap \borel\right) &\leq C_1 \beta^\alpha \mu(B) & \text{(quasi-decaying)}
\end{align}
or
\begin{align}
\label{weakQDwithE}
\mu\left(\thickvar{\LL}{\beta\|\dist_\LL\|_{\mu,B}}\cap B\cap \borel\right) &\leq C_1 \beta^\alpha \mu(B) & \text{(weakly quasi-decaying)},
\end{align}
respectively. We will say that $\mu$ is \emph{(weakly) quasi-decaying relative to $\borel$} if for $\mu$-a.e. $\xx\in \borel$, $\mu$ is (weakly) quasi-decaying at $\xx$ relative to $\borel$. Finally, we will say that $\mu$ is \emph{(weakly) quasi-decaying} if there exists a sequence $(\borel_\Index)_\Index$ such that $\mu\left(\R^d\butnot\bigcup_\Index \borel_\Index\right) = 0$ and for each $\Index$, $\mu$ is (weakly) quasi-decaying relative to $\borel_\Index$.
\end{definition}

Let us briefly discuss several aspects of Definition \ref{definitionQD} which differ from Definition \ref{definitionabsolutelyfriendly}:% Let us briefly discuss several aspects etc
\begin{itemize}
\item The uniform dependence of the constants $C_1$ and $\alpha$ on the point $\xx$ has been dropped. Moreover, the condition is only required to hold for $\mu$-a.e. every $\xx$, rather than for all $\xx$ in the support of $\mu$. This makes the quasi-decay conditions closer to the ``non-uniform'' versions of friendliness considered in \cite[\66]{KLW}. By itself, this relaxation does not seem to give any natural new examples of measures satisfying the condition, until it is combined with the other relaxations considered below.
\item The left-hand sides of \eqref{QDwithE} and \eqref{weakQDwithE} include an intersection with a set $\borel$ which has large but not full measure with respect to $\mu$. This change is done for two reasons:
\begin{itemize}
\item It makes quasi-decay into a measure class invariant. Note that the relaxation of uniformity is not itself enough to make the condition a measure class invariant (see Theorem \ref{theoremmeasureclass} below).
\item Sequences $(\borel_\Index)_\Index$ with the property described above often show up naturally in our proofs (see e.g. Theorem \ref{theoremexactdim}).
\end{itemize}
\item The inequalities \eqref{QDwithE} and \eqref{weakQDwithE} are only required to hold for $0 < \beta \leq \rho^\gamma$, rather than for all $\beta > 0$. This is probably the most unexpected aspect of our definition. It means that as the balls $B = B(\xx,\rho)$ get smaller, the thicknesses of hyperplane-neighborhoods whose measures can be bounded in terms of $\mu(B)$ get smaller not only in an absolute sense, but also relative to the radius $\rho$.
\item The Federer (doubling) condition has been dropped. The reason for this is that there is an analogue of the Federer condition (Lemma \ref{lemmaquasifederer}) which is good enough for our purposes and which holds for every measure on every doubling metric space, and in particular for every measure on $\R^d$.
\item The nonplanarity condition has been incorporated directly into the definition of weak quasi-decay by using closed thickenings rather than open ones. This difference is mathematically insignificant, but it is a notational convenience.
\end{itemize}

It is obvious that the following implications hold:
\begin{center}
\begin{tabular}{|ccc|}
\hline
Absolutely friendly & \implies & Friendly\\
$\Downarrow$ & & $\Downarrow$\\
Quasi-decaying & \implies & Weakly quasi-decaying\\
\hline
\end{tabular}
\end{center}
(The strictness of these implications is shown by examples in \cite[Figure 1]{DFSU_GE2}.) Moreover, the appropriate analogues of the friendliness/absolute friendliness relations hold:
\begin{itemize}
\item[(i)] every quasi-decaying measure is weakly quasi-decaying;
\item[(ii)] the Lebesgue measure of a nondegenerate submanifold of $\R^d$ is weakly quasi-decaying but not quasi-decaying;
\item[(iii)] more generally, the image of a quasi-decaying measure under a nondegenerate embedding is weakly quasi-decaying; more precisely:
\end{itemize}
\begin{theorem}[Proven in Section \ref{sectionbasic}]
\label{theoremQDembedding}
For all $\ell\in\N$ and $\epsilon > 0$, the image of a quasi-decaying measure under an $\ell$-nondegenerate embedding of class $\CC^{\ell + \epsilon}$ is weakly quasi-decaying.
\end{theorem}
In relation to extremality, we shall prove that every weakly quasi-decaying measure is strongly extremal (Corollary \ref{corollaryKLW}), thus generalizing the main result of \cite{KLW} and in particular providing a third proof of Sprind\v zuk's conjecture. This implication also proves a conjecture of KLW \cite[\610.5]{KLW} that nonplanar and decaying measures are strongly extremal, i.e. that the Federer condition is unnecessary in their main theorem. Although the proof of this result uses essentially the full machinery of the existing proofs of Sprind\v zuk's conjecture \cite{KleinbockMargulis,KLW}, it is worth noting that the following result (which does not imply Sprind\v zuk's conjecture) can be proven using only elementary real analysis together with the Simplex Lemma:
\begin{theorem}[Proven in Section \ref{sectionQDE}]
\label{theoremQDE}
Every quasi-decaying measure is extremal.
\end{theorem}

The idea of proving the extremality of measures using the Simplex Lemma is due to A. D. Pollington and S. L. Velani \cite[Theorem 1]{PollingtonVelani}. Proving Theorem \ref{theoremQDE} was a key step in our construction of the definition of the quasi-decay condition, since it allowed us to see what the minimal hypotheses on the measure were such that the proof would work. It was only later that we realized the Sprind\v zuk conjecture machinery developed in \cite{KleinbockMargulis, KLW} would work for our measures as well.

\subsection{Ahlfors regularity vs. exact dimensionality}
One way of thinking about the difference between KLW's conditions and our conditions is by comparing this difference with the difference between the classes of \emph{Ahlfors regular} and \emph{exact dimensional} measures, both of which are well-studied in dynamics (for more details see \cite{DFSU_GE2}). We recall their definitions:

\begin{definition*}
A measure $\mu$ on $\R^d$ is called \emph{Ahlfors $\delta$-regular} if there exists $C > 0$ such that for every ball $B(\xx,\rho)$ with $\xx\in\Supp(\mu)$ and $0 < \rho \leq 1$.
\[
C^{-1} \rho^\delta \leq \mu\big(B(\xx,\rho)\big) \leq C \rho^\delta.
\]
The measure $\mu$ is called \emph{exact dimensional of dimension $\delta$} if for $\mu$-a.e. $\xx\in\R^d$,
\begin{equation}
\label{exactdim}
\lim_{\rho\to 0} \frac{\log\mu\big(B(\xx,\rho)\big)}{\log\rho} = \delta.
\end{equation}
\end{definition*}

Every Ahlfors $\delta$-regular measure is exact dimensional of dimension $\delta$. The Hausdorff and packing dimensions of an exact dimensional measure of dimension $\delta$ are both equal to $\delta$ \cite[Theorem 4.4]{Young2}; for an Ahlfors $\delta$-regular measure, the Hausdorff, packing, and upper and lower Minkowski (box-counting) dimensions of the topological support are also equal to $\delta$. There are many dynamical examples of Ahlfors regular measures; there are also many examples of exact dimensional measures which are not Ahlfors regular. In Part II, the latter class of examples will prove to be a fruitful source of quasi-decaying measures which are not friendly.

The philosophical relations between Ahlfors regularity and exact dimensionality with absolute friendliness and quasi-decay, respectively, are:
\begin{equation}
\label{philosophical}
\begin{split}
\text{Ahlfors regular and ``nonplanar''} \;\; &\Rightarrow \;\; \text{Absolutely friendly}\\
\text{Exact dimensional and ``nonplanar''} \;\; &\Rightarrow \;\; \text{Quasi-decaying}
\end{split}
\end{equation}
Here ``nonplanar'' does not refer to nonplanarity as defined in Definition \ref{definitionabsolutelyfriendly}, but is rather something less precise (and stronger). This less precise definition should rule out examples like the Lebesgue measures of nondegenerate manifolds, since these are not quasi-decaying. One example of a ``sufficient condition'' for this imprecise notion of ``nonplanarity'' is simply the inequality $\delta > d - 1$, where $\delta$ is the dimension of the measure in question. In particular, in this context the relations \eqref{philosophical} are made precise by the following theorems:

\begin{theorem*}[{\cite[Proposition 6.3]{KleinbockWeiss1}}; cf. \cite{PollingtonVelani,Urbanski3}]
If $\delta > d - 1$, then every Ahlfors $\delta$-regular measure on $\R^d$ is absolutely friendly.
\end{theorem*}

\begin{theorem}[Proven in Section \ref{sectionQDE}]
\label{theoremexactdim}
If $\delta > d - 1$, then every exact dimensional measure on $\R^d$ of dimension $\delta$ is quasi-decaying.
\end{theorem}

\subsection{Further comparison of KLW's conditions vs. our conditions}
There are three axes on which we can compare our conditions against KLW's: Diophantine properties of measures satisfying the condition, examples of measures satisfying the condition, and stability properties. We deal with the first of these in \6\ref{subsectionadditional} below, and the second will be discussed at length in Part II. It remains to consider stability properties. The following proposition describes the stability properties of quasi-decaying and weakly quasi-decaying measures:

\begin{theorem}[Proven in Section \ref{sectionbasic}]
\label{theoremstabilityQD}~
\begin{itemize}
\item[(i)] The (weak) quasi-decay property does not depend on which norm $\|\cdot\|$ on $\R^d$ is used in Definition \ref{definitionQD}.
\item[(ii)] The product of any two (weakly) quasi-decaying measures is (weakly) quasi-decaying.
\item[(iii)] For all $\epsilon > 0$, the image of a quasi-decaying measure under a $\CC^{1 + \epsilon}$ diffeomorphism is quasi-decaying.
\item[(iv)] If $(U_i)_i$ is an open cover of $\R^d$, then $\mu$ is (weakly) quasi-decaying if and only if for each $i$, $\mu\given U_i$ is (weakly) quasi-decaying.
\item[(v)] Any measure absolutely continuous with respect to a (weakly) quasi-decaying measure is (weakly) quasi-decaying.
\end{itemize}
\end{theorem}

The first two properties are also satisfied for friendliness and absolute friendliness (for (ii) see \cite[Theorem 2.4]{KLW}). Property (iii) is not true for either friendliness or the weakly quasi-decay condition, since the image of the Lebesgue measure of a nondegenerate manifold under a diffeomorphism may be the Lebesgue measure of an affine hyperplane, which does not satisfy any of the four conditions (if the hyperplane is rational it is not even extremal). It is true for the absolute friendliness condition under the additional hypothesis that the measure is compactly supported \cite[Proposition 3.2]{FMS}.

Property (iv) is not true for either friendliness or absolute friendliness, but this can be fixed either by making a more careful statement which involves conditions holding ``relative to'' certain open sets in the sense of Definition \ref{definitionabsolutelyfriendly}, or else by considering ``non-uniform'' versions of the conditions, as is done in \cite[\66]{KLW}. It was hypothesized in \cite[para. after Theorem 6.1]{KLW} that a weak version of property (v) holds for the non-uniform versions of friendliness and absolute friendliness, namely that these conditions are measure class invariants. However, we can now show that this statement is false; see Appendix \ref{appendix1}.

We remark that stability properties (iii) and (iv) imply that it makes sense to talk about quasi-decaying measures on abstract differentiable manifolds, by calling a measure quasi-decaying if it is quasi-decaying on every coordinate chart. The non-uniform version of absolute friendliness can be also considered on manifolds. It doesn't make sense to talk about weakly quasi-decaying or friendly measures on abstract differentiable manifolds due to the failure of property (iii) for these classes.

\subsection{Additional Diophantine properties of quasi-decaying measures}
\label{subsectionadditional}
In addition to being extremal, the Lebesgue measure of a nondegenerate manifold has many other nice Diophantine properties which can also be generalized to weakly quasi-decaying measures. These improvements fall into three categories:
\begin{itemize}
\item those dealing with strong extremality rather than just extremality;
\item those dealing with matrices rather than just vectors;
\item those dealing with measures supported on proper affine subspaces of $\R^d$ (or in the case of matrices, of the space $\EE$ defined below).
\end{itemize}
Let us review the theory of Diophantine approximation of matrices. Fix $\bothdim\in\N$, let $\MM \equiv \MM_\bothdim$\Footnote{Here and elsewhere $A \equiv B$ means ``$A$ is shorthand for $B$''.} denote the set of $\pdim\times \qdim$ matrices, and fix $\bfA\in\MM$. Rather than approximating $\bfA$ by rational matrices, classically one considers ``approximations'' to $\bfA$ to be integer vectors $\qq\in \Z^\qdim\butnot\{\0\}$ whose image under $\bfA$ is close to an integer vector. Thus the \emph{exponent of irrationality} of $\bfA$ is defined as
\[
\omega(\bfA) = \limsup_{\substack{\qq\in\Z^\qdim\butnot\{\0\} \\ \pp\in\Z^\pdim}} \frac{-\log\|\bfA\qq - \pp\|}{\log\|\qq\|},
\]
where the two $\|\cdot\|$s denote any two norms on $\R^\pdim$ and $\R^\qdim$, and the \emph{exponent of multiplicative irrationality} is the number\Footnote{This definition agrees with the multiplicative approximation framework considered in \cite{KMW}, but not the one considered in \cite{KleinbockMargulis}; see comments after Proposition \ref{propositiondynamicalinterpretation} for more details.}
\[
\omega_\mult(\bfA) = \limsup_{\substack{\qq\in\Z^\qdim\butnot\{\0\} \\ \pp\in\Z^\pdim}} \frac{-\log\prod_{i = 1}^\pdim |(\bfA\qq - \pp)_i|}{\log \prod_{j = 1}^\qdim |q_j|\vee 1}\cdot
\]
Note that $\omega_\mult(\bfA) \geq (\pdim/\qdim)\omega(\bfA)$. The relationship between matrix approximation and simultaneous approximation (i.e. the approximation of vectors in $\R^d$ by rational vectors described at the beginning of this paper) is as follows: if $\qdim = 1$ and $\xx = \bfA\ee_1$, then $\omega(\bfA) = \omega(\xx) - 1$ and $\omega_\mult(\bfA) = \omega_\mult(\xx) - \pdim$. The matrix $\bfA$ is called \emph{very well approximable} if $\omega(\bfA) > \qdim/\pdim$, and \emph{very well multiplicatively approximable} if $\omega_\mult(\bfA) > 1$. As in the case of vectors, we denote the set of very well (multiplicatively) approximable $\pdim\times \qdim$ matrices by $\mathrm{VW(M)A}_{\bothdim}$, and we call a measure $\mu$ on $\MM$ \emph{extremal} if $\mu(\VWA_{\bothdim}) = 0$ and \emph{strongly extremal} if $\mu(\VWMA_{\bothdim}) = 0$. Also as before, the sets $\VWA_{\bothdim}$ and $\VWMA_{\bothdim}$ are both Lebesgue nullsets of full Hausdorff dimension which satisfy $\VWA_{\bothdim} \subset \VWMA_{\bothdim}$.

It turns out (cf. \cite{KMW, BKM, ABRdS,ABRdS3}) that the natural vector space structure of $\MM$ is not appropriate for determining extremality and strong extremality. Instead, it is better to identify $\MM$ with its image under the \emph{Pl\"ucker embedding} $\psi \equiv \psi_{\bothdim}:\MM_{\bothdim} \to \EE \equiv \EE_{\bothdim}$, where $\EE\subset \bigwedge^\qdim \R^{\pdim + \qdim}$ is the subspace spanned by all basis vectors (vectors of the form $\ee_I = \bigwedge_{i\in I} \ee_i$\, where the product is taken in increasing order) other than $\bigwedge_{j = 1}^\qdim (\0\oplus\ee_j)$, and
\begin{equation}
\label{plucker}
\psi(\bfA) = \bigwedge_{j = 1}^\qdim (\bfA\ee_j\oplus\ee_j) - \bigwedge_{j = 1}^\qdim (\0\oplus\ee_j) \in \EE.
\end{equation}
Concretely, $\psi$ is the map which sends a matrix to the list of the determinants of its minors.

\begin{remark*}
Technically, the map $\psi$ defined by \eqref{plucker} is not the Pl\"ucker embedding, but is related to it as follows. Let $\GG \equiv \GG(\qdim,\pdim + \qdim)$ denote the Grassmannian space consisting of all $\qdim$-dimensional subspaces of $\R^{\pdim + \qdim}$, and let $\PP$ denote the projectivization of the vector space $\bigwedge^\qdim \R^{\pdim + \qdim}$. Consider the coordinate charts $\iota_1:\MM\to \GG$ and $\iota_2:\EE\to\PP$ defined by the formulas $\iota_1(\bfA) = (\bfA \oplus I_\qdim)(\R^\qdim)$ and $\iota_2(\omega) = (\bigwedge_1^\qdim (\0\oplus\ee_j) + \omega)\R$. Then
\[
\iota_2\circ \psi = \wbar\psi\circ \iota_1,
\]
where $\wbar\psi:\GG\to\PP$ is the true Pl\"ucker embedding. Nevertheless, we shall continue to call the map defined by \eqref{plucker} the Pl\"ucker embedding.
\end{remark*}

Given a measure $\mu$ on $\MM$, we can ask about its geometric properties (e.g. friendliness, quasi-decay) either with respect to the natural vector space structure on $\MM$ or with respect to the natural identification of $\MM$ with a submanifold of $\EE$ via the Pl\"ucker embedding. When $\qdim = 1$ or $\pdim = 1$, the Pl\"ucker embedding is a linear isomorphism, so the geometric properties of $\mu$ do not depend on which way we consider $\MM$. In general, these properties may depend on which way we consider $\MM$, but due to the nondegeneracy of the Pl\"ucker embedding, the following relations hold (cf. \cite[Theorem 2.1]{KLW} and Theorem \ref{theoremQDembedding}):
\begin{itemize}
\item If $\mu$ is absolutely friendly with respect to the vector space structure of $\MM$, then $\mu$ is friendly when considered as a measure on $\EE$.
\item If $\mu$ is quasi-decaying with respect to the vector space structure of $\MM$, then $\mu$ is weakly quasi-decaying when considered as a measure on $\EE$.
\end{itemize}
In such a scenario, the following theorem implies that $\mu$ is strongly extremal:

\begin{theorem}[Corollary of Theorem \ref{theoremexponents} below]
\label{theoremKLW}
Let $\mu$ be a measure on $\MM$ which is weakly quasi-decaying when considered as a measure on $\EE$. Then $\mu$ is strongly extremal.
\end{theorem}
The special case of Theorem \ref{theoremKLW} which occurs when $\mu$ is friendly instead of weakly quasi-decaying was proven in \cite[Theorem 2.1]{KMW}.

Combining with Theorem \ref{theoremQDembedding} yields:
\begin{corollary}
\label{corollaryKLW}
Let $\mu$ be a measure on $\MM$ which is quasi-decaying with respect to the vector space structure of $\MM$. Then $\mu$ is strongly extremal. If $\pdim = 1$ or $\qdim = 1$, $\mu$ need only be weakly quasi-decaying.
\end{corollary}
Note that Corollary \ref{corollaryKLW} provides an alternate proof of Theorem \ref{theoremQDE}.

Although a measure $\mu$ supported on an affine subspace of $\EE$ cannot be weakly quasi-decaying, if $\mu$ is weakly quasi-decaying with respect to the affine subspace, then we can get information about the exponent of irrationality:

\begin{theorem}[Proven in Section \ref{sectionKLW}]
\label{theoremexponents}
Let $\mu$ be a measure on $\MM$ which is supported on an affine subspace $\AA\subset\EE$ and which is weakly quasi-decaying when interpreted as a measure on $\AA$. Then for $\mu$-a.e. $\bfA\in\MM$
\begin{align}
\label{exponents}
\omega(\bfA) &= \inf\{\omega(\bfB) : \bfB\in\MM\cap \AA\}.
\end{align}
Moreover, $\mu$ is strongly extremal if and only if $\MM\cap\AA\nsubset \VWMA$.
\end{theorem}
Note that Theorem \ref{theoremKLW} follows from Theorem \ref{theoremexponents} by taking $\AA = \EE$, since it is well-known that in this case the right hand side of \eqref{exponents} is $\pdim/\qdim$, and that $\MM \nsubset \VWMA$. It appears to be difficult to prove a multiplicative analogue of \eqref{exponents}, due to difficulties with providing a dynamical interpretation for the exponent of multiplicative irrationality function; cf. Footnote \ref{footnotetechnical}.

Historical note: A special case of \eqref{exponents}, where the condition of being weakly quasi-decaying is replaced by an analogue of a friendliness condition, was proven independently by Aka, Breuillard, Rosenzweig, and de Saxc\'e \cite[Theorem 5.2.5]{ABRdS3} (see also their earlier announcement of this result in \cite[Theorem 4.3]{ABRdS}). Their paper also contains other interesting information about the function $\AA\mapsto \inf\{\omega(\bfB) : \bfB\in\MM\cap \AA\}$, such as its value when $\AA$ is rational.

\subsection{An overview of Part II}\Footnote{In this subsection we refer to the references cited for the definitions of terms used in the theorems.}
\label{subsectionpreview}
Theorem \ref{theoremexactdim} (exact dimensional measures of sufficiently large dimension are quasi-decaying) already provides large classes of examples of quasi-decaying measures which are not known to be friendly. For example, the following result was proven by F. Hofbauer:

\begin{theorem*}[{\cite[Theorem 1]{Hofbauer}}]
Let $T:[0,1]\to[0,1]$ be a piecewise monotonic transformation whose derivative has bounded $p$-variation for some $p > 0$. Let $\mu$ be a measure on $[0,1]$ which is ergodic and invariant with respect to $T$. Let $h(\mu)$ and $\chi(\mu)$ denote the entropy and Lyapunov exponent of $\mu$, respectively. If $\chi(\mu) > 0$, then $\mu$ is exact dimensional of dimension
\[
\delta_\mu := \frac{h(\mu)}{\chi(\mu)}\cdot
\]
\end{theorem*}

Note that if $h(\mu) > 0$, then Ruelle's inequality \cite[Theorem 7.1]{BarrioJimenez} implies that $\chi(\mu) > 0$, so the above result applies and gives $\delta_\mu > 0 = d - 1$, so $\mu$ is quasi-decaying, and in particular extremal.\Footnote{The inequality $\chi(\mu) < \infty$ follows from the hypothesis that $T'$ has bounded $p$-variation, which in particular implies that $T'$ is bounded.}

There are numerous other classes of measures coming from dynamics which are known to be exact dimensional. A notable example is the theorem of Barreira, Pesin, and Schmeling \cite{BPS} to the effect that any measure ergodic, invariant, and hyperbolic with respect to a diffeomorphism is exact dimensional. Theorem \ref{theoremexactdim} applies directly to those measures whose dimension is sufficiently large, but in Part II we will mostly be interested in the question of what happens for measures whose dimension is not large enough. (We will also be interested in measures which are not necessarily exact dimensional but which nevertheless can be proved to be quasi-decaying.) As mentioned above, the philosophy is that some sort of ``nonplanarity'' assumption should be able to substitute for the large-dimension hypothesis. For inspiration we can turn to the known dynamical examples of absolutely friendly measures \cite{KLW, StratmannUrbanski1, Urbanski}, which share the property that the nonplanarity hypothesis takes the form: the dynamical system in question cannot preserve a manifold of strictly lower dimension than the ambient space.

Our next examples of quasi-decaying measures are generalizations of the known examples of absolutely friendly measures. For example, the following theorem generalizes the main result of \cite{Urbanski}:

\begin{theorem}[{\cite[Theorem 1.14]{DFSU_GE2}}]
\label{theoremgibbs}
Fix $d\in\N$, and let $(u_a)_{a\in A}$ be a (possibly infinite) irreducible conformal iterated function system (CIFS) on $\R^d$. Let $\phi:A^\N\to \R$ be a summable locally H\"older continuous potential function, let $\mu_\phi$ be a Gibbs measure of $\phi$, and $\pi:A^\N\to\R^d$ be the coding map. Suppose that the Lyapunov exponent
\begin{equation}
\label{lyapunov}
\chi_{\mu_\phi} = \int \log(1/|u_{\omega_1}'(\pi\circ\sigma(\omega))|) \;\; \dee\mu_\phi(\omega)
\end{equation}
is finite. Then $\pi_*[\mu_\phi]$ is quasi-decaying.
\end{theorem}

This theorem generalizes \cite{Urbanski} in two different ways:
\begin{itemize}
\item The CIFS can be infinite, as long as the Lyapunov exponent is finite.
\item The open set condition is no longer needed.
\end{itemize}

Note that if $\phi$ is the ``conformal potential'' $\phi(\omega) = -\log |u_{\omega_1}'(\pi(\sigma(\omega)))|$, then the convergence of \eqref{lyapunov} for some $\alpha$ is equivalent to the strong regularity of the CIFS $(u_a)_{a\in E}$. Thus the following is a corollary of Theorem \ref{theoremgibbs}:

\begin{corollary}[Conformal measures of infinite iterated function systems]
Fix $d\in\N$, and let $(u_a)_{a\in E}$ be a strongly regular conformal iterated function system acting irreducibly on an open set $W\subset \R^d$. Let $\mu$ be the conformal measure of $(u_a)_{a\in E}$. Then $\mu$ is quasi-decaying.
\end{corollary}

Our next example extends the result of \cite{StratmannUrbanski1} from the setting of convex-cocompact groups to the setting of geometrically finite groups:

\begin{theorem}[{\cite[Theorems 1.9 and 1.17]{DFSU_GE2}}]
Let $G$ be a geometrically finite group of M\"obius transformations of $\R^d$ which does not preserve any generalized sphere. Then the Patterson--Sullivan measure $\mu$ of $G$ is both quasi-decaying and friendly. However, $\mu$ is absolutely friendly if and only if every cusp of $G$ has maximal rank.
\end{theorem}

An interesting aspect of this example is that we are able to prove the extremality of the Patterson--Sullivan measure using KLW's condition; for this particular example it was not necessary to introduce the quasi-decay condition. However, proving quasi-decay has the advantage of also proving that the measure is extremal with respect to matrix approximations as well; cf. \6\ref{subsectionadditional} above.

In subsequent papers, we plan to find sufficient conditions for quasi-decay for many other classes of measures as well, but at this stage we cannot give precise theorem statements.

On the other hand, it is also interesting to consider dynamical measures which are not extremal. Three of the authors have already considered this question in \cite{FSU1}, where the following was proven:

\begin{theorem}[{\cite[Theorem 4.5]{FSU1}}]
\label{theoremFSU}
There exists a measure $\mu$ invariant with respect to the Gauss map which gives full measure to the Liouville numbers. In particular, $\mu$ is not extremal.
\end{theorem}

By \cite[Theorem 2.1]{FSU1}, the measure $\mu$ in Theorem \ref{theoremFSU} must have infinite Lyapunov exponent. In Part II, we show that for certain dynamical systems (namely hyperbolic toral endomorphisms), the class of invariant measures which give full measure to the Liouville points is not only nonempty but topologically generic:

\begin{theorem}[{\cite[Theorem 1.13]{DFSU_GE2}}]
Let $T:X\to X$ be a hyperbolic toral endomorphism, where $X = \R^d/\Z^d$. Let $\M_T(X)$ be the space of $T$-invariant probability measures on $X$. Then the set of measures which give full measure to the Liouville points is comeager in $\M_T(X)$.
\end{theorem}

{\bf Outline of the paper.}
In Section \ref{sectionQDE} we give elementary arguments proving that every exact-dimensional measure of dimension $>d - 1$ is quasi-decaying (Theorem \ref{theoremexactdim}), and that every quasi-decaying measure is extremal (Theorem \ref{theoremQDE}). In Section \ref{sectionbasic} we demonstrate the basic properties of the quasi-decay condition described in Theorems \ref{theoremQDembedding} and \ref{theoremstabilityQD}. In Section \ref{sectionKLW} we prove Theorem \ref{theoremexponents}, describing the Diophantine properties of weakly quasi-decaying measures with respect to matrix approximation.

\draftnewpage
\section{Proof of Theorems \ref{theoremQDE} and \ref{theoremexactdim} ($\delta > d - 1 \;\Rightarrow\text{ Quasi-decaying }\Rightarrow\text{ Extremal}$)}
\label{sectionQDE}

\begin{definition}
\label{definitionUQD}
Given a measure $\mu$ on $\R^d$ and a set $\borel\subset\R^d$, we will say that $\mu$ is \emph{uniformly quasi-decaying (resp. uniformly weakly quasi-decaying)} relative to $\borel$ if for all $\gamma > 0$, there exist $C_1,\alpha > 0$ such that for all $\xx\in \borel$, $0 < \rho \leq 1$, $0 < \beta \leq \rho^\gamma$, and $\LL\in\Hyp$, if $B = B(\xx,\rho)$ then \eqref{QDwithE} (resp. \eqref{weakQDwithE}) holds.
\end{definition}

\begin{lemma}
\label{lemmaunifQD}
A measure $\mu$ is (weakly) quasi-decaying if and only if there exists a sequence $(\borel_\Index)_\Index$ such that $\mu(\R^d\butnot\bigcup_\Index \borel_\Index) = 0$ and for each $\Index$, $\mu$ is uniformly (weakly) quasi-decaying relative to $\borel_\Index$.
\end{lemma}
\begin{proof}
It suffices to show that if $\mu$ is (weakly) quasi-decaying relative to $\borel$, then there exists a sequence $(\borel_\Index)_\Index$ such that $\mu(\borel\butnot\bigcup_\Index \borel_\Index) = 0$ and for each $\Index$, $\mu$ is uniformly (weakly) quasi-decaying relative $\borel_\Index$. Indeed, for each $m,k\in\N$ let $\borel_{m,k}$ be the set of all $\xx\in \borel$ such that \eqref{QDwithE} (resp. \eqref{weakQDwithE}) holds for all $0 < \rho\leq 1$, $0 < \beta \leq \rho^\gamma$, and $\LL\in\Hyp$, with $\gamma = 1/m$, $\alpha = 1/k$, and $B = B(\xx,\rho)$. Then for all $m$, $\mu(\borel\butnot\bigcup_k \borel_{m,k}) = 0$, so there exists $k_m\in\N$ such that $\mu(\borel\butnot \borel_{m,k_m}) \leq 2^{-m}$. Letting
\[
\borel_\Index \df \bigcap_{m > \Index} \borel_{m,k_m}
\]
completes the proof.
\end{proof}

Actually, uniformly quasi-decaying measures show up naturally in the analysis of exact dimensional measures:

\begin{proof}[Proof of Theorem \ref{theoremexactdim}]
Let $\mu$ be an exact dimensional measure on $\R^d$ of dimension $\delta > d - 1$, and we will show that $\mu$ is quasi-decaying. By Egoroff's theorem, there exists a sequence $(\borel_\Index)_\Index$ such that $\mu(\R^d\butnot \bigcup \borel_\Index) = 0$ and for all $\Index\in\N$, the limit \eqref{exactdim} holds uniformly on $\borel_\Index$. Fix $\Index$, and we will show that $\mu$ is uniformly quasi-decaying relative to $\borel_\Index$. Indeed, fix $\gamma > 0$, $\xx\in \borel_\Index$, $0 < \rho \leq 1$, $0 < \beta \leq \rho^\gamma$, and $\LL\in\Hyp$. Let $(\xx_i)_1^N$ be a maximal $\beta\rho$-separated\Footnote{Recall that a set $S$ is said to be \emph{$\rho$-separated} if for all distinct $x,y\in S$, we have $\dist(x,y)\geq \rho$.} subset of $\thickvar\LL{\beta\rho}\cap B(\xx,\rho)\cap \borel_\Index$, and let $\lambda$ denote Lebesgue measure on $\R^d$. Then
\[
N (\beta\rho)^d \asymp_\times \sum_{i = 1}^N \lambda\big(B(\xx_i,\beta\rho/2)\big)
\leq \lambda\left(\thickvar\LL{2\beta\rho}\cap B(\xx,2\rho)\right)
\asymp_\times \rho^{d - 1} (\beta\rho) = \beta\rho^d,
\]
so $N\lesssim_\times \beta^{-(d - 1)}$. On the other hand, for all $\epsilon > 0$ we have
\begin{align*}
\mu\big(\thickvar{\LL}{\beta\rho}\cap B(\xx,\rho)\cap \borel_\Index\big)
&\leq_\pt \sum_{i = 1}^N \mu\big(B(\xx_i,\beta\rho)\big)\\
&\lesssim_\times N (\beta\rho)^{\delta - \epsilon} \by{\eqref{exactdim}}\\
&\lesssim_\times \beta^{-(d - 1) + (\delta - \epsilon)}\rho^{s - \epsilon}\\
&\lesssim_\times \beta^{\delta - (d - 1) - \epsilon} \rho^{-2\epsilon} \mu\big(B(\xx,\rho)\big). \by{\eqref{exactdim}}
\end{align*}
Letting $\alpha = \delta - (d - 1) > 0$ and $\epsilon = \alpha/(1 + 2/\gamma) > 0$, since $\rho \geq \beta^{1/\gamma}$ we get
\[
\mu\big(\thickvar{\LL}{\beta\rho}\cap B(\xx,\rho)\cap \borel_\Index\big)
\lesssim_\times \beta^{\alpha/2} \mu\big(B(\xx,\rho)\big).
\qedhere\]
\end{proof}

Next, we prove Theorem \ref{theoremQDE}. By Lemma \ref{lemmaunifQD}, it suffices to demonstrate the following:

\begin{theorem}
Let $\mu$ be a measure which is uniformly quasi-decaying relative to a set $\borel\subset\R^d$. Then 
\[
\mu(\VWA_d\cap \borel) = 0 .
\]
\end{theorem}
\begin{proof}
For each $\gamma > 0$ let
\[
W_\gamma \df \{\xx\in\R^d: \omega(\xx) > (1 + 1/d)(1 + \gamma)\},
\]
so that $\VWA_d = \bigcup_{\gamma > 0} W_\gamma$. Fix $\gamma > 0$, and we will show that $\mu(W_\gamma\cap \borel) = 0$. Since $\mu$ is uniformly quasi-decaying relative to $\borel$, there exist $C_1,\alpha > 0$ such that for all $\xx\in \borel$, $0 < \rho \leq 1$, $0 < \beta \leq \rho^\gamma$, and $\LL\in\Hyp$, if $B = B(\xx,\rho)$ then \eqref{QDwithE} holds.

To proceed further we recall the \emph{simplex lemma}, which is proven by a volume argument:

\begin{lemma}[{\cite[Lemma 4]{KTV}}]
\label{lemmasimplex}
Fix $d\in\N$. There exists $\epsilon_d > 0$ such that for all $\yy\in\R^d$ and $0 < \rho\leq 1$, the set
\[
S_{\yy,\rho} \df \left\{\frac{\pp}{q}\in\Q^d\cap B(\yy,\rho): q\leq \epsilon_d \rho^{-d/(d + 1)}\right\}
\]
is contained in an affine hyperplane $\LL_{\yy,\rho} \subset\R^d$.
\end{lemma}

\noindent Let $\epsilon_d > 0$ be as in Lemma \ref{lemmasimplex}. Fix $H > 1$, and for each $\Index\in\N$ let
\begin{align*}
Q_\Index &\df \epsilon_d H^{d\Index},&
\rho_\Index &\df \frac12H^{-(d + 1)\Index}.
\end{align*}
For each $\Index\in\N$ and $\yy\in\R^d$ let
\begin{align*}
S_{\Index,\yy} &\df S_{\yy,2\rho_\Index} = \left\{\frac{\pp}{q}\in\Q^d\cap B(\yy,2\rho_\Index): q\leq Q_\Index\right\},&
\LL_{\Index,\yy} &\df \LL_{\yy,2\rho_\Index}.
\end{align*}
Fix $\Index\in\N$, and let $E_\Index\subset W_\gamma\cap \borel$ be a maximal $\rho_\Index$-separated set.
\begin{claim}
\label{claimSbound}
\[
W_\gamma\cap \borel \subset \limsup_{\Index\to\infty}\bigcup_{\yy\in E_\Index}\left[\thickvar{\LL_{\Index,\yy}}{\rho_\Index^{1 + \gamma}}\cap B(\yy,\rho_\Index)\right].
\]
\end{claim}
\begin{subproof}
Fix $\xx\in W_\gamma\cap \borel$ and $\pp/q\in\Q^d$, and let $\Index\in\N$ satisfy $Q_{\Index - 1} \leq q < Q_\Index$. Then
\[
q^{-(1 + 1/d)(1 + \gamma)} \leq Q_{\Index - 1}^{-(1 + 1/d)(1 + \gamma)} \asymp_\times \rho_\Index^{1 + \gamma}
\]
and so since $\omega(\xx) > (1 + 1/d)(1 + \gamma)$, there exist infinitely many $\pp/q$ such that
\begin{equation}
\label{goodpq}
\|\xx - \pp/q\| < \rho_\Index^{1 + \gamma} < \rho_\Index.
\end{equation}
Fix $\pp/q$ satisfying \eqref{goodpq}. Since $\xx\in W_\gamma\cap \borel$, there exists $\yy\in E_\Index$ such that $\xx\in B(\yy,\rho_\Index)$. Then $\pp/q\in S_{\Index,\yy} \subset \LL_{\Index,\yy}$ and thus by \eqref{goodpq}, we have $\xx\in \thickvar{\LL_{\Index,y}}{\rho_\Index^{1 + \gamma}}$.
\end{subproof}

Without loss of generality we may assume that $\borel$ is bounded, so that $\mu\big(\thickvar \borel1\big) < \infty$. Then for each $\Index\in\N$, we have
\begin{align*}
&\sum_{\yy\in E_\Index} \mu\left(\thickvar{\LL_{\Index,\yy}}{\rho_\Index^{1 + \gamma}}\cap B(\yy,\rho_\Index)\cap \borel\right) \noreason \\
&\leq_\pt C_1 \rho_\Index^{\gamma\alpha} \sum_{\yy\in E_\Index} \mu\big(B(\yy,\rho_\Index)\big) \note{uniform quasi-decay}\\
&\lesssim_\times \rho_\Index^{\gamma\alpha} \mu\big(\thickvar \borel1\big) \asymp_\times \rho_\Index^{\gamma\alpha}. \note{bounded multiplicity}
\end{align*}
So by the Borel-Cantelli lemma,
\[
\mu\left(\limsup_{\Index\to\infty}\bigcup_{\yy\in E_\Index}\left[\thickvar{\LL_{\Index,\yy}}{\rho_\Index^{1 + \gamma}}\cap B(\yy,\rho_\Index)\cap \borel\right]\right) = 0,
\]
and so $\mu(W_\gamma\cap \borel) = 0$ by Claim \ref{claimSbound}.
\end{proof}

\section{Basic properties of the quasi-decay condition}
\label{sectionbasic}
Before proving Theorems \ref{theoremQDembedding} and \ref{theoremstabilityQD}, we need some preliminaries. The first, as mentioned in the introduction, is a substitute for the doubling condition which holds for every measure on a doubling metric space.

\begin{definition}
A metric space $X$ is \emph{doubling} if there exists a constant $N_X$ such that every ball $B(x,\rho) \subset X$ can be covered by at most $N_X$ balls of radius $\rho/2$.
\end{definition}

For example, $\R^d$ is a doubling metric space.

\begin{lemma}
\label{lemmaquasifederer}
Let $X$ be a doubling metric space, and let $\mu$ be a measure on $X$. Then for all $\epsilon > 0$, there exists $\delta > 0$ such that for $\mu$-a.e. $x\in X$, there exists $C_2 > 0$ such that for all $0 < \rho \leq 1$,
\begin{equation}
\label{quasifederer}
\mu\big(B(x,\rho^{1 - \delta})\big) \leq C_2 \rho^{-\epsilon} \mu\big(B(x,\rho)\big).
\end{equation}
\end{lemma}
\begin{proof}
Fix $\epsilon > 0$, and let $\delta = \epsilon/(2\log_2(N_X)) > 0$, where $N_X$ is the doubling constant of $X$. For each $\Index\in\N$ let $\rho_\Index = 2^{-\Index}$ and let
\[
S_\Index \df \{x\in X: \mu\big(B(x,\rho_\Index^{1 - \delta})\big) > \rho_\Index^{-\epsilon}\mu\big(B(x,\rho_{\Index + 1})\big)\}.
\]
\begin{claim}
If $\borel\subset X$ is bounded then $\sum_{\Index\in\N}\mu(S_\Index\cap \borel) < \infty$.
\end{claim}
\begin{subproof}
Fix $\Index$, and let $\borel_\Index$ be a maximal $\rho_{\Index + 1}$-separated subset of $S_\Index\cap \borel$, so that $S_\Index\cap \borel\subset \thickvar{\borel_\Index}{\rho_{\Index + 1}}$. We have
\begin{align*}
\mu(S_\Index\cap \borel) &\leq \sum_{x\in \borel_\Index} \mu\big(B(x,\rho_{\Index + 1})\big)\\
&\leq \rho_\Index^\epsilon \sum_{x\in \borel_\Index} \mu\big(B(x,\rho_\Index^{1 - \delta})\big)\\
&= \rho_\Index^\epsilon \int \#(\borel_\Index\cap B(x,\rho_\Index^{1 - \delta})) \dee\mu(x)\\
&\leq \rho_\Index^\epsilon \mu\big(\thickvar \borel 1\big) \max_{x\in X} \#(\borel_\Index\cap B(x,\rho_\Index^{1 - \delta})).
\end{align*}
Fix $x\in X$. Repeatedly applying the doubling condition shows that $B(x,\rho_\Index^{1 - \delta})$ can be covered by at most $N_X^m$ balls of radius $\rho_{\Index + 1}/3$, where $m \df \lceil \log_2(6\rho_\Index^{-\delta})\rceil$. But each of these balls intersects $\borel_\Index$ at most once, since $\borel_\Index$ is $\rho_{\Index + 1}$-separated. So
\[
\max_{x\in X} \#(\borel_\Index\cap B(x,\rho_\Index^{1 - \delta})) \leq N_X^m \asymp_\times N_X^{\log_2(\rho_\Index^{-\delta})} = \rho_\Index^{-\epsilon/2}.
\]
Thus $\mu(S_\Index) \lesssim_\times \rho_\Index^{\epsilon/2}$.
\end{subproof}
So by the Borel--Cantelli lemma, for $\mu$-a.e. $x\in X$ we have $\#\{\Index\in\N:x\in S_\Index\} < \infty$. Fix such an $x$, and fix $0 < \rho \leq 1$. Let $\Index\in\N$ satisfy $\rho_{\Index + 1} \leq \rho < \rho_\Index$. If $\rho$ is small enough, then $x\notin S_\Index$, which implies
\[
\mu\big(B(x,\rho^{1 - \delta})\big) \leq \mu\big(B(x,\rho_\Index^{1 - \delta})\big) \leq \rho_\Index^{-\epsilon} \mu\big(B(x,\rho_{\Index + 1})\big) \leq 2^\epsilon \rho^{-\epsilon} \mu\big(B(x,\rho)\big),
\]
demonstrating \eqref{quasifederer}. Larger values of $\rho$ can be accomodated by changing the constant appropriately.
\end{proof}

Let us call a measure satisfying the conclusion of Lemma \ref{lemmaquasifederer} \emph{quasi-Federer}, so that Lemma \ref{lemmaquasifederer} says that any measure on a doubling metric space is quasi-Federer. For the purposes of this paper this is a somewhat silly definition, since every measure on $\R^d$ is quasi-Federer. However, the following refinements of the quasi-Federer notion distinguish nontrivial classes of measures on $\R^d$:

\begin{definition}
Let $X$ and $\mu$ be as in Lemma \ref{lemmaquasifederer}. Given $\borel\in X$, we will say that $\mu$ is \emph{uniformly quasi-Federer relative to $\borel$} if for all $\epsilon > 0$, there exist $C_2,\delta > 0$ such that for all $x\in\borel$ and $0 < \rho \leq 1$, \eqref{quasifederer} holds. (Note however that $\borel$ does not occur on the left hand side of \eqref{quasifederer}, in contrast to \eqref{QDwithE}.) Similarly, if $x\in X$, we will say that $\mu$ is \emph{quasi-Federer at $x$} if $\mu$ is uniformly quasi-Federer relative to $\{x\}$.

Note that Lemma \ref{lemmaquasifederer} implies that there exists a sequence of sets $(\borel_\Index)_\Index$ such that $\mu\big(X \butnot \bigcup \borel_\Index\big) = 0$ and for each $\Index$, $\mu$ is uniformly quasi-Federer relative to $\borel_\Index$. In particular, $\mu$ is quasi-Federer at $\mu$-a.e. $x\in X$.
\end{definition}

We need two more preliminary results. The following lemma is an immediate consequence of Definitions \ref{definitionQD} and \ref{definitionUQD}:
\begin{lemma}
\label{lemmaQD}
~
\begin{itemize}
\item[(i)] If $\mu$ is uniformly quasi-decaying (resp. uniformly weakly quasi-decaying) relative to $\borel\subset\R^d$, then for all $\gamma > 0$ there exists $\alpha = \alpha(\gamma,\mu) > 0$ such that for all $C > 0$, there exists $C_1 > 0$ such that for all $\xx \in \borel$, $0 < \rho \leq 1$, $\beta \leq C \rho^\gamma$, and $\LL\in\Hyp$, if $B = B(\xx,\rho)$ then \eqref{QDwithE} (resp. \eqref{weakQDwithE}) holds.
\item[(ii)] If $\mu$ is quasi-decaying (resp. weakly quasi-decaying) at $\xx\in\R^d$ relative to $\borel\subset\R^d$, then for all $\gamma > 0$ there exists $\alpha = \alpha(\gamma,\mu,\xx) > 0$ such that for all $C > 0$, there exists $C_1 > 0$ such that for all $0 < \rho \leq 1$, $\beta \leq C \rho^\gamma$, and $\LL\in\Hyp$, if $B = B(\xx,\rho)$ then \eqref{QDwithE} (resp. \eqref{weakQDwithE}) holds.
\end{itemize}
\end{lemma}

Our last preliminary result is a generalization of the Lebesgue differentiation theorem:

\begin{theorem}[{\cite[Theorem 9.1]{Simmons1}}]
\label{theoremLD}
Let $\mu$ and $\nu$ be measures on $\R^d$ such that $\nu\lessless\mu$. Then the function
\begin{equation}
\label{lebesguederivative}
f(x) \df \lim_{\rho\to 0}\frac{\nu\big(B(x,\rho)\big)}{\mu\big(B(x,\rho)\big)}
\end{equation}
is well-defined for $\mu$-almost every $\xx\in \R^d$. Moreover, $\nu = f\mu$, i.e. $f$ is a Radon-Nikodym derivative of $\nu$ with respect to $\mu$.
\end{theorem}

We are now ready to prove Theorem \ref{theoremstabilityQD}; it clearly follows from Lemma \ref{lemmaquasifederer} and Theorem \ref{theoremLD} together with the following:

\begin{proposition}
\label{propositionstabilityQD2}~
\begin{itemize}
\item[(i)] Let $\|\cdot\|_1$ and $\|\cdot\|_2$ be two norms on $\R^d$. If $\mu$ is (weakly) quasi-decaying and quasi-Federer at $\xx\in\R^d$ relative to $\borel\subset\R^d$ with respect to the norm $\|\cdot\|_1$, then $\mu$ is also (weakly) quasi-decaying at $\xx$ relative to $\borel$ with respect to the norm $\|\cdot\|_2$.
\item[(ii)] For each $i = 1,2$, fix $d_i\in\N$, and let $\mu_i$ be a measure on $\R^{d_i}$ which is (weakly) quasi-decaying and quasi-Federer at a point $\xx_i\in\R^{d_i}$ relative to a set $\borel_i\subset\R^{d_i}$. Let $d = d_1 + d_2$. Then $\mu = \mu_1\times \mu_2$ is (weakly) quasi-decaying at $\xx = (\xx_1,\xx_2)\in \R^d$ relative to $\borel = \borel_1\times \borel_2 \subset \R^d$.
\item[(iii)] Fix $\epsilon > 0$. Let $\mu_1$ be a measure on an open set $U_1\subset\R^d$ which is uniformly quasi-decaying relative to a set $\borel_1\subset U_1$. Let $\psi:U_1\to U_2\subset\R^d$ be a $\CC^{1 + \epsilon}$ diffeomorphism. Then if $\mu_2 = \psi(\mu_1)$ is quasi-Federer at $\xx_2\in U_2$, then $\mu_2$ is also quasi-decaying at $\xx_2$ relative to $\borel_2 = \psi(\borel_1)$.
\item[(iv)] Let $\mu$ be a measure on $\R^d$, and let $U\subset\R^d$ be an open set. Then $\mu$ is (weakly) quasi-decaying at a point $\xx\in U$ relative to a set $\borel\subset\R^d$ if and only if $\mu\given U$ is (weakly) quasi-decaying at $\xx$ relative to $\borel$.
\item[(v)] Let $\mu$ be a measure on $\R^d$ which is (weakly) quasi-decaying at a point $\xx\in\R^d$ relative to a set $\borel\subset\R^d$, and let $\nu$ satisfy $\nu\lesssim_\times \mu$ on $\borel$. If the limit \eqref{lebesguederivative} exists and is positive, then $\nu$ is (weakly) quasi-decaying at $\xx$ relative to $\borel$.
\end{itemize}
\end{proposition}

\begin{proof}~
\begin{itemize}
\item[(i)]
Let $C > 0$ be the implied constant in the asymptotic $\|\cdot\|_1 \asymp_\times \|\cdot\|_2$, which holds because any two norms on $\R^d$ are equivalent. Fix $\gamma > 0$, and let $\alpha \df \alpha(\gamma,\mu,\xx) > 0$ be as in Lemma \ref{lemmaQD}. Fix $0 < \rho\leq 1$, $0 < \beta \leq \rho^\gamma$, and $\LL\in\Hyp$. Note that
\[
\|d_\LL^{(2)}\|_{\mu,B_2(\xx,\rho)} \leq C\|d_\LL^{(1)}\|_{\mu,B_1(\xx,C\rho)},
\]
where $B_i(\xx,\rho)$ denotes the ball $B(\xx,\rho)$ taken with respect to the norm $\|\cdot\|_i$, and similarly for $d_\LL^{(i)}$. If $\mu$ is weakly quasi-decaying, then
\begin{align*}
&\mu\big(\NN_2(\LL,\beta\|d_\LL^{(2)}\|_{\mu,B_2(\xx,\rho)})\cap B_2(\xx,\rho)\cap \borel\big) \noreason\\
&\leq_\pt \mu\big(\NN_1(\LL,\beta C^2\|d_\LL^{(1)}\|_{\mu,B_1(\xx,C\rho)})\cap B_1(\xx,C\rho)\cap \borel\big) \noreason\\
&\lesssim_\times \beta^\alpha\mu\big(B_1(\xx,C\rho)\big) \by{Lemma \ref{lemmaQD}}\\
&\lesssim_\times \beta^{\alpha/2} \mu\big(B_2(\xx,\rho)\big)/ \since{$\mu$ is quasi-Federer at $\xx$}
\end{align*}
If $\mu$ is quasi-decaying, then a similar argument shows that
\[
\mu\big(\NN_2(\LL,\beta\rho)\cap B_2(\xx,\rho)\cap \borel\big) \lesssim_\times \beta^{\alpha/2} \mu\big(B_2(\xx,\rho)\big).
\]
\item[(ii)]
By part (i), we can use any norm on $\R^d = \R^{d_1}\times \R^{d_2}$ in the proof. It is convenient to use the max norm $\|\cdot\|_\infty$, so that $B(\xx,\rho) = B(\xx_1,\rho)\times B(\xx_2,\rho)$ for all $\rho > 0$. Fix $\gamma > 0$, let $\alpha_i = \alpha(\gamma,\mu_i,\xx_i)$ be as in Lemma \ref{lemmaQD}, and let $\alpha = \alpha_1\wedge\alpha_2 > 0$. Fix $0 < \rho \leq 1$, $0 < \beta \leq \rho^\gamma$, and $\LL\in\Hyp$. Write $B_i = B(\xx_i,\rho)$ and $B = B_1\times B_2$. There exist $\zz = (\zz_1,\zz_2)\in\R^d\butnot\{\0\}$ and $c\in\R$ such that
\[
\LL = \left\{\yy\in\R^d: \zz\cdot\yy = c \right\}.
\]
Without loss of generality suppose $\|\zz\|_1 = 1$. For each $\yy_1\in\R^{d_1}$, let
\[
\LL_{\yy_1} = \{\yy_2\in\R^{d_2} : (\yy_1,\yy_2)\in\LL\} = \{\yy_2\in\R^{d_2} : \zz_1\cdot \yy_1 + \zz_2\cdot\yy_2 = c\}.
\]
Note that for all $\yy = (\yy_1,\yy_2)\in\R^d$,
\begin{align} \label{dLz}
d_\LL(\yy) &= |\zz\cdot\yy - c|\\ \label{dLyz}
d_{\LL_{\yy_1}}(\yy_2) &= |\zz_1\cdot\yy_1 + \zz_2\cdot\yy_2 - c||/\|\zz_2\|_1.
\end{align}
In particular
\begin{equation}
\label{comparison}
d_{\LL_{\yy_1}}(\yy_2) = d_\LL(\yy_1,\yy_2) / \|\zz_2\|_1.
\end{equation}
We divide into cases:\\
\item[(iia)] Quasi-decaying case. Since $\|\zz_1\|_1 + \|\zz_2\|_1 = \|\zz\|_1 = 1$, there exists $i = 1,2$ such that $\|\zz_i\|_1\geq 1/2$. Without loss of generality, suppose $\|\zz_2\|_1\geq 1/2$. Then
\begin{align*}
&\mu\big(\thickvar\LL{\beta\rho}\cap B\cap \borel\big)\\
&=_\pt \int_{B_1\cap \borel_1} \mu_2\big(\big\{\yy_2\in B_2\cap \borel_2 : d_\LL(\yy_1,\yy_2)\leq \beta\rho\big\}\big) \;\dee\mu_1(\yy_1)\noreason\\
&\leq_\pt \int_{B_1} \mu_2\big(\thickvar{\LL_{\yy_1}}{2\beta\rho} \cap B_2\cap \borel_2\big) \;\dee\mu_1(\yy_1) \by{\eqref{comparison}}\\
&\lesssim_\times \beta^\alpha \int_{B_1} \mu_2(B_2) \;\dee\mu_1(\yy_1) = \beta^\alpha \mu(B). \note{Lemma \ref{lemmaQD}}
\end{align*}
Thus $\mu$ is quasi-decaying at $\xx$ relative to $\borel$.
\item[(iib)] Weakly quasi-decaying case. Let $\sigma = \|d_\LL\|_{\mu,B}$. We can assume that
\[
\thickvar\LL{(1/3)\sigma}\cap B\cap\Supp(\mu) \neq \emptyset,
\]
as otherwise \eqref{weakQDwithE} holds trivially. Then there exist $\aa,\bb\in B\cap\Supp(\mu)$ such that $d_\LL(\aa) \leq (1/3)\sigma \leq (2/3)\sigma \leq d_\LL(\bb)$. So by \eqref{dLz},
\[
|\zz\cdot \bb - \zz\cdot\aa| \geq (1/3)\sigma.
\]
Without loss of generality, we may suppose that
\[
|\zz_2\cdot\bb_2 - \zz_2\cdot\aa_2| \geq (1/6)\sigma.
\]
Then for all $\yy_1\in \R^d$, by \eqref{dLyz} we have
\[
d_{\LL_{\yy_1}}(\aa) + d_{\LL_{\yy_1}}(\bb) \geq (1/6)\sigma/\|\zz_2\|_1
\]
and thus
\[
\|d_{\LL_{\yy_1}}\|_{\mu_2,B_2} \geq (1/12)\sigma/\|\zz_2\|_1.
\]
Applying \eqref{comparison} gives
\[
\{\yy_2\in\R^{d_2}: (\yy_1,\yy_2)\in\thickvar\LL{\beta\sigma}\}
\subset \thickvar{\LL_{\yy_1}}{12\beta\|d_{\LL_{\yy_1}}\|_{\mu_2,B_2}},
\]
so
\begin{align*}
&\mu\big(\thickvar\LL{\beta\sigma}\cap B\cap \borel\big)\\
&=_\pt \int_{B_1\cap \borel_1} \mu_2\big(\big\{\yy_2\in B_2\cap \borel_2 : d_\LL(\yy_1,\yy_2)\leq \beta\sigma\big\}\big) \;\dee\mu_1(\yy_1)\noreason\\
&\leq_\pt \int_{B_1} \mu_2\big(\thickvar{\LL_{\yy_1}}{12\beta\|d_{\LL_{\yy_1}}\|_{\mu_2,B_2}} \cap B_2\cap \borel_2\big) \;\dee\mu_1(\yy_1) \by{\eqref{comparison}}\\
&\lesssim_\times \beta^\alpha \int_{B_1} \mu_2(B_2) \;\dee\mu_1(\yy_1) = \beta^\alpha \mu(B). \note{Lemma \ref{lemmaQD}}
\end{align*}
\item[(iii)]
The proof of (iii) is similar to the proof of Proposition \ref{propositionQDembedding} below. More precisely, in that proof we can replace $\|d_\LL\|_{\mu_2,B_2}$ by $\|d_\LL\|_{U_2\cap B_2}$ without affecting the argument. Here $\|d_\LL\|_B \df \sup_B \dist(\cdot,\LL)$. Since $U_2$ is open, for $\rho$ sufficiently small we have $\|d_\LL\|_{U_2\cap B_2} = \|d_\LL\|_{B_2} \geq \rho$. Thus in this case, the proof of Proposition \ref{propositionQDembedding} actually proves \eqref{QDwithE} rather than just \eqref{weakQDwithE}.
\item[(iv)]
This is immediate upon changing the implied constant of \eqref{QDwithE} or \eqref{weakQDwithE} appropriately to handle $0 < \rho \leq 1$ for which $B(\xx,\rho)\nsubset U$.
\item[(v)]
If the limit \eqref{lebesguederivative} exists and is positive, then $\mu(B(\xx,\rho)) \asymp_\times \nu\big(B(\xx,\rho)\big)$ for all $0 < \rho \leq 1$. The claim follows immediately.
\qedhere
\end{itemize}
\end{proof}

We now prepare for the proof of Theorem \ref{theoremQDembedding}. The key idea, already implicitly contained in the proofs of \cite[Theorem 7.6]{KLW} and \cite[Theorem 4.6]{FKMS}, is to cover the neighborhood of the zero set of a smooth function by neighborhoods of hyperplanes. We bring this idea to the foreground by stating the following lemma, in which we use the notation
\begin{align*}
\|f\|_B &\df \sup_{\xx\in B} |f(\xx)|,&
\|f\|_{\CC^\epsilon,B} &\df \sup_{\xx,\yy\in B} \frac{|f(\yy) - f(\xx)|}{\|\yy - \xx\|^\epsilon},&
\Delta &\df B(\0,1) \subset \R^d,\\
\|f\| &\df \|f\|_\Delta,&
\|f\|_{\CC^\epsilon} &\df \|f\|_{\CC^\epsilon,\Delta}
\end{align*}

\begin{lemma}
\label{lemmaZPbeta}
Fix $\ell\in\N$ and $0 < \epsilon\leq 1$, and let $f:\Delta\to\R$ be a function of class $\CC^{\ell + \epsilon}$ such that
\begin{equation}
\label{fellbound}
\|f^{(\ell)}\|_{\CC^\epsilon} \leq \kappa_\ell \|f\|,
\end{equation}
where $\kappa_\ell > 0$ is a small constant depending on $\ell$ and $\epsilon$. Then for all $\beta > 0$ sufficiently small (depending on $\ell$ and $\epsilon$), the set
\[
\ZZ(f,\beta) \df \{\xx\in\Delta : |f(\xx)| \leq \beta\|f\|\}
\]
can be covered by collections $\CC_1,\ldots,\CC_\ell$, where for each $k = 1,\ldots,\ell$, the collection $\CC_k$ takes the form
\begin{equation}
\label{Ck}
\CC_k \df \{\thickvar{\LL_j}{\beta_k^{1 + \epsilon/2}}\cap B(\pp_j,\beta_k) : j\in J_k\},
\end{equation}
where $\beta_k \df \beta^{1/2^{2k - 1}}$, $(\pp_j)_{j\in J_k}$ is a $\beta_k$-separated sequence in $\Delta$, and $(\LL_j)_{j\in J_k}$ is a sequence of affine hyperplanes.
\end{lemma}
\begin{proof}
We proceed by induction on $\ell$. If $\ell = 0$, then we let $\kappa_0 = 1/2$, which implies that $\ZZ(f,\beta) = \emptyset$ for all $\beta < 1/2$ and so the lemma is trivial. So suppose that $\ell \geq 1$ and $\ZZ(f,\beta)\neq\emptyset$. Then 
\[
\inf_\Delta |f| \leq \beta \|f\| \leq (1/2) \|f\|,
\] so by the mean value inequality, there exists $i = 1,\ldots,d$ such that $\|\del_i f\| \gtrsim_\times \|f\|$. Let $C_1 > 0$ denote the implied constant and let $\kappa_\ell \leq \kappa_{\ell - 1}/C_1$. Then
\begin{align*}
\|\del_i f^{(\ell - 1)}\|_{\CC^\epsilon}
&\leq \|f^{(\ell)}\|_{\CC^\epsilon}
\leq \kappa_\ell \|f\| \leq C_1\kappa_\ell \|\del_i f\| \leq \kappa_{\ell - 1} \|\del_i f\|,
\end{align*}
so by the induction hypothesis, $\ZZ(\del_i f,\beta^{1/4})$ can be covered by collections $\CC_2,\ldots,\CC_\ell$ of the form \eqref{Ck}. If $\ell = 1$, then by the base case of the induction we have $\ZZ(\del_i f, \beta^{\epsilon/4}) = \emptyset$ assuming $\beta$ is sufficiently small. Let
\[
\gamma = \begin{cases}
\epsilon & \ell = 1\\
1 & \ell \geq 2
\end{cases},
\]
so that either way, $\ZZ(\del_i f,\beta^{\gamma/4})$ can be covered by collections $\CC_2,\ldots,\CC_\ell$ of the form \eqref{Ck}. So to complete the proof, we need to cover $\ZZ(f,\beta)\butnot \ZZ(\del_i f,\beta^{\gamma/4})$ by a collection $\CC_1$ of the form $\eqref{Ck}_{k = 1}$. Let $(\pp_j)_{j\in J_1}$ be a maximal $\beta_1 = \beta^{1/2}$-separated sequence in $\Delta\butnot \ZZ(\del_i f,\beta^{\gamma/4})$, and let $J = J_1$. Fix $j\in J$, so that 
\[
|\del_i f(\pp_j)| > \beta^{\gamma/4} \|\del_i f\| \asymp_\times \beta^{\gamma/4} \|f\|.
\] 
Let $B_j \df B(\pp_j,\beta_1)$.
\begin{claim}
For all $\yy\in B_j\cap \Delta$,
\[
|f(\yy) - f(\pp_j) - f'(\pp_j)[\yy - \pp_j]| \lesssim_\times \beta_1^{1 + \gamma} \|f\|.
\]
\end{claim}
\begin{subproof}
Elementary calculus gives
\begin{align*}
|f(\yy) - f(\pp_j) - f'(\pp_j)[\yy - \pp_j]|
&\leq \|\yy - \pp_j\| \sup_{\zz\in B_j\cap \Delta} \|f'(\zz) - f'(\pp_j)\|.
\end{align*}
Since $\|\yy - \pp_j\| \leq \beta_1$, to complete the proof we need to show that
\begin{equation}
\label{ETSZP}
\|f'(\zz) - f'(\pp_j)\| \lesssim_\times \beta_1^\gamma \|f\| \all\zz\in B_j\cap \Delta.
\end{equation}
If $\ell = 1$, then \eqref{ETSZP} follows directly from \eqref{fellbound}. So suppose that $\ell \geq 2$, and write $f = P + R$, where $P$ is the Taylor polynomial of $f$ at $\0$ of order $\ell$. By \eqref{fellbound} and the mean value inequality we have
\begin{equation}
\label{Rbounds}
\|R\| \lesssim_\times \cdots \lesssim_\times \|R^{(\ell)}\| \leq \|R^{(\ell)}\|_{\CC^\epsilon} \lesssim_\times \kappa_\ell \|f\|,
\end{equation}
so by making $\kappa_\ell$ sufficiently small we can guarantee that $\|R\| \leq (1/2)\|f\|$ and thus $\|f\| \asymp_\times \|P\|$. But $\|P\|$ is asymptotic to the maximum of the coefficients of $P$, which implies that $\|P^{(2)}\| \lesssim_\times \|P\|$. On the other hand, $\|R^{(2)}\| \lesssim_\times \|f\|$ by \eqref{Rbounds}, so overall we have $\|f^{(2)}\| \lesssim_\times \|f\|$. Applying the mean value inequality yields \eqref{ETSZP}.
\end{subproof}
Thus if we let $\LL_j = \{\yy : f(\pp_j) + f'(\pp_j)[\yy - \pp_j] = 0\}$, then for all $\yy\in B_j\cap \Delta$ we have
\[
\dist(\yy,\LL_j) = \frac{|f(\pp_j) - f'(\pp_j)[\yy - \pp_j]|}{\|f'(\pp_j)\|} \lesssim_\times \frac{|f(\yy)| + \beta_1^{1 + \gamma} \|f\|}{\beta_1^{\gamma/2}\|f\|}\cdot
\]
So for $\yy\in \ZZ(f,\beta)\cap B_j\cap \Delta$, we have $|f(\yy)|\leq \beta\|f\| \leq \beta_1^{1 + \gamma} \|f\|$ and thus $\dist(\yy,\LL_j) \lesssim_\times \beta_1^{1 + \gamma/2}$. So if $\beta$ is small enough, then
\[
B_j\cap \Delta\cap \ZZ(f,\beta) \subset \thickvar{\LL_j}{\beta_1^{1 + \gamma/3}}.
\]
Taking the union over $j\in J = J_1$ gives
\[
\Delta\cap \ZZ(f,\beta)\butnot \ZZ(\del_i f,\beta^{\gamma/4}) \subset \bigcup_{j\in J_1} \thickvar{\LL_j}{\beta_1^{1 + \gamma/3}}\cap B_j = \bigcup(\CC_1).
\qedhere\]
\end{proof}

We are almost ready to prove Theorem \ref{theoremQDembedding}. First, we recall the definition of a nondegenerate embedding:

\begin{definition}
\label{definitionnondegenerate}
Let $U\subset\R^d$ be an open set, and let $\psi:U\to\R^D$ be a map of class $\CC^1$. Suppose that $\psi$ is a \emph{smooth embedding}, i.e. that $\psi$ is a homeomorphism onto its image and that for each $\xx\in U$, the linear transformation $\psi'(\xx)$ is injective. Given $\xx\in U$ and $\ell\in\N$, we say that $\psi$ is \emph{$\ell$-nondegenerate at $\xx$} if $\psi$ is of class $\CC^\ell$ in a neighborhood of $\xx$ and
\[
\R^D = \psi'(\xx)[\R^d] + \psi''(\xx)[\R^d\otimes\R^d] + \cdots + \psi^{(\ell)}(\xx)[(\R^d)^{\otimes \ell}].
\]
If $\psi$ is $\ell$-nondegenerate at every point of $U$ (resp. at almost every point of $U$), then we say that $\psi$ is $\ell$-nondegenerate (resp. $\ell$-weakly nondegenerate), or just nondegenerate (resp. weakly nondegenerate). The manifold $\psi(U)$ will also be called $\ell$-nondegenerate (resp. $\ell$-weakly nondegenerate).
\end{definition}

It is not hard to see that if $U$ is connected and $\psi$ is a real-analytic smooth embedding, then $\psi$ is weakly nondegenerate if and only if $\psi(U)$ is not contained in any affine hyperplane. Even in the setting of smooth maps, examples of connected smooth embeddings which are strongly degenerate (i.e. not weakly nondegenerate) but not contained in any affine hyperplane are somewhat pathological \cite{Wolsson}.

Theorem \ref{theoremQDembedding} now follows from Lemmas \ref{lemmaunifQD} and \ref{lemmaquasifederer} together with the following:

\begin{proposition}
\label{propositionQDembedding}
Let $\mu_1$ be a measure on an open set $U\subset\R^d$, let $\borel_1\subset U$, and suppose that $\mu_1$ is uniformly quasi-decaying relative to $\borel_1$. Fix $\ell\in\N$ and $\epsilon > 0$, and let $\psi:U\to\R^D$ be a smooth embedding which is $\ell$-nondegenerate at a point $\xx_1\in U$, and of class $\CC^{\ell + \epsilon}$ in a neighborhood of $\xx_1$. Then if $\mu_2 = \psi(\mu_1)$ is quasi-Federer at $\xx_2 = \psi(\xx_1)$, then $\mu_2$ is weakly quasi-decaying at $\xx_2$ relative to $\borel_2 = \psi(\borel_1)$.
\end{proposition}
\begin{proof}
Fix $\gamma > 0$, $0 < \rho\leq 1$, $0 < \beta \leq \rho^\gamma$, $\LL\in\Hyp$, and $B_2 \df B(\xx_2,\rho)$. Since $\psi$ is a smooth embedding, for some constant $C_1 > 0$ we have $\psi(B_1)\supset \psi(U)\cap B_2$, where $B_1 \df B(\xx_1,C_1\rho)$. Let $\pi:\R^d\to\R$ be an affine map such that for all $\yy\in\R^d$, $\dist(\yy,\LL) = |\pi(\yy)|$. Then
\begin{align*}
\mu_2\big(\thickvar{\LL}{\beta\|\dist_\LL\|_{\mu_2,B_2}}\cap B_2\cap \borel_2\big)
&\leq \mu_1\big(\big\{\yy\in B_1\cap \borel_1 : |\pi\circ\psi(\yy)| \leq \beta \|\pi\circ\psi\|_{B_1}\big\}\big).
\end{align*}
Let $T(\zz) = \xx_1 + C_1\rho\zz$, so that $T(\Delta) = B_1$. Let $f = \pi\circ\psi\circ T$, so that
\[
\big\{\yy\in B_1\cap \borel_1 : |\pi\circ\psi(\yy)| \leq \beta \|\pi\circ\psi\|_{B_1}\big\}
= T\big(\big\{\zz\in\Delta : |f(\zz)| \leq \beta \|f\|\big\}\big)\cap \borel_1.
\]
Let $\bfP:\R^d\to\R^D$ be the Taylor approximation of $\psi$ at $\xx_1$ to order $\ell$, and let $P = \pi\circ\bfP\circ T$. Since $\psi$ is of class $\CC^{\ell + \epsilon}$ in a neighborhood of $\xx_1$, we have
\[
\|f - P\| \lesssim_\times \rho^{\ell + \epsilon}.
\]
On the other hand, since by hypothesis $\bfP(\R^d)$ is not contained in any affine hyperplane, a compactness argument shows that $\|\pi\circ\bfP\|\asymp_\times 1$, and thus
\[
\|P\| \gtrsim_\times \rho^\ell \|\pi\circ\bfP\| \asymp_\times \rho^\ell.
\]
Thus if $\rho$ is sufficiently small, then $\|f\| \asymp_\times \|P\|$. We also have
\[
\|f^{(\ell)}\|_{\CC^\epsilon}
= \rho^{\ell + \epsilon} \|(\pi\circ\psi)^{(\ell)}\|_{\CC^\epsilon,B_1} \leq \rho^{\ell + \epsilon} \|\psi^{(\ell)}\|_{\CC^\epsilon,B_1} \lesssim_\times \rho^{\ell + \epsilon},
\]
so if $\rho$ is sufficiently small then \eqref{fellbound} holds. Let the collections $\CC_1,\ldots,\CC_\ell$ be given by Lemma \ref{lemmaZPbeta}. For each $k = 1,\ldots,\ell$ let $\beta_k$, $(\pp_j)_{j\in J_k}$, and $(\LL_j)_{j\in J_k}$ be as in \eqref{Ck}. Then
\begin{equation}
\label{whereweleftoff}
\begin{split}
&\mu_2\big(\thickvar{\LL}{\beta\|\dist_\LL\|_{\mu_2,B_2}}\cap B_2\cap \borel_2\big)\\
&\leq \mu_1\left(\bigcup_{k = 1}^\ell \bigcup_{j\in J_k} T\big(\thickvar{\LL_j}{\beta_k^{1 + \epsilon/2}}\cap B(\pp_j,\beta_k)\big)\cap \borel_1\right)\\
&\leq \sum_{k = 1}^\ell \sum_{j\in J_k} \mu_1\big(\thickvar{T(\LL_j)}{C_1 \beta_k^{1 + \epsilon/2}\rho}\cap B(T(\pp_j),C_1\beta_k \rho)\cap \borel_1\big)\\
&= \sum_{k = 1}^\ell \sum_{j\in J_k} \mu_1\big(\thickvar{T(\LL_j)}{\beta_k^{\epsilon/2}\rho_k}\cap B(T(\pp_j),\rho_k/2)\cap \borel_1\big),
\end{split}
\end{equation}
where $\rho_k = 2C_1 \beta_k \rho$.

Fix $k = 1,\ldots,\ell$ and $j\in J_k$. We claim that for some $\alpha > 0$ depending only on $\gamma$,
\begin{equation}
\label{Txy}
\mu_1\big(\thickvar{T(\LL_j)}{\beta_k^{\epsilon/2}\rho_k}\cap B(T(\pp_j),\rho_k/2)\cap \borel_1\big) \lesssim_\times \beta^\alpha \mu_1\big(B(T(\pp_j),2\rho_k)\big).
\end{equation}
To avoid trivialities, assume that the set in the left hand side is nonempty, and let $\yy$ be a member of that set. Then $B(T(\pp_j),\rho_k/2)\subset B(\yy,\rho_k) \subset B(T(\pp_j),2\rho_k)$, so it is enough to show that
\[
\mu_1\big(\thickvar{T(\LL_j)}{\beta_k^{\epsilon/2}\rho_k}\cap B(\yy,\rho_k)\cap \borel_1\big) \lesssim_\times \beta^\alpha \mu_1\big(B(\yy,\rho_k)\big).
\]
Write $\delta_k \df 1/2^{2k - 1}$, so that $\beta_k \df \beta^{\delta_k}$. Then
\begin{align*}
\beta_k &\leq \rho^{\gamma\delta_k},&
\rho_k &\lesssim_\times \beta_k^{1 + 1/(\gamma\delta_k)},&
\beta_k^{\epsilon/2} &\lesssim_\times \rho_k^{\frac{\epsilon/2}{1 + 1/(\gamma\delta_k)}}.
\end{align*}
Let $\alpha_k = \alpha\big(\frac{\epsilon/2}{1 + 1/(\gamma\delta_k)},\mu_1\big)$ (cf. Lemma \ref{lemmaQD}). Then since $\yy\in \borel_1$, we have
\[
\mu_1\big(\thickvar{T(\LL_j)}{\beta_k^{\epsilon/2}\rho_k}\cap B(\yy,\rho_k)\cap \borel_1\big) \lesssim_\times \beta_k^{\alpha_k \epsilon/2} \mu_1\big(B(\yy,\rho_k)\big),
\]
and letting $\alpha \df \min_k \delta_k \alpha_k \epsilon/2 > 0$ completes the proof of \eqref{Txy}.

We finish the proof with the following calculation:
\begin{align*}%\ddr
&\mu_2\big(\thickvar{\LL}{\beta\|\dist_\LL\|_{\mu_2,B_2}}\cap B_2\cap \borel_2\big)\\
&\lesssim_\times \sum_{k = 1}^\ell \sum_{j\in J_k} \beta^\alpha \mu_1\big(B(T(\pp_j),2\rho_k)\big) \by{\eqref{whereweleftoff} and \eqref{Txy}}\\
&\asymp_\times \beta^\alpha \sum_{k = 1}^\ell \mu_1\big(B(\xx_1,2 C_1 \rho)\big) \note{bounded multiplicity}
\end{align*}
\begin{align*}
&\lesssim_\times \beta^{\alpha} \mu_2\big(B(\xx_2,2C_2 \rho)\big) \note{for some $C_2 > 0$}\\
&\lesssim_\times \beta^{\alpha/2} \mu_2\big(B(\xx_2,\rho)\big). \since{$\mu_2$ is quasi-Federer at $\xx_2$}
&\qedhere\end{align*}
\end{proof}

\draftnewpage
\section{More refined Diophantine properties of quasi-decaying measures}
\label{sectionKLW}

In this section, we fix $\bothdim\in\N$, and let $\MM = \MM_\bothdim$ denote the set of $\pdim\times \qdim$ matrices as in \6\ref{subsectionadditional}. We will usually identify $\MM$ with its image under the Pl\"ucker embedding $\psi:\MM\to \EE$ defined by \eqref{plucker}; however, we will sometimes distinguish between $\bfA\in\MM$ and $\psi(\bfA)\in\EE$ for clarity. Our main goal in this section is to prove Theorem \ref{theoremexponents}, using the techniques of \cite{KleinbockMargulis2, KLW, BKM, KMW}. We begin by introducing a uniform framework with which to talk about exponents of irrationality and their multiplicative versions. Our tool for doing this is the Dani--Kleinbock--Margulis correspondence principle between Diophantine approximation and the dynamics of homogeneous flows \cite{Dani3, KleinbockMargulis}.

\subsection{The correspondence principle}
\label{subsectioncorrespondence}
To start with, we introduce the notations
\begin{align*}
\Lambda_0 &\df \Z^{\pdim + \qdim}\\
u_\bfA &\df \left[\begin{array}{cc}
I_\pdim & \bfA \\
& I_\qdim
\end{array}\right] & (\bfA\in\MM)\\
g_\tt &\df \left[\begin{array}{ccc}
e^{t_1} &&\\
& \ddots &\\
&& e^{t_{\pdim + \qdim}}
\end{array}\right] & (\tt\in\mfa),
\end{align*}
where $\mfa \df \{\tt\in\R^{\pdim + \qdim} : \sum t_i = 0\}$. Next let
\begin{align*}
\mfa_+ &\df \left\{\tt\in \mfa :
\begin{array}{c}
\text{$t_i\leq 0$ for $i \leq \pdim$}\\
\text{$t_i\geq 0$ for $i > \pdim$}
\end{array}\right\}\\
\mfa_+^* &\df \left\{\tt\in \mfa_+ :
\begin{array}{c}
t_1 = \cdots = t_\pdim\\
t_{\pdim + 1} = \cdots = t_{\pdim + \qdim}
\end{array}
\right\}.
\end{align*}
Finally, given $\SS\subset\mfa_+$ and a function $\norm:\SS\to\Rplus$, we let
\[
\omega(\bfA;\SS,\norm) \df \limsup_{\SS\ni\tt\to\infty} \frac{1}{\norm(\tt)}\Delta\big(g_\tt u_\bfA \Lambda_0\big),
\]
where
\[
\Delta(\Lambda) \df -\log\min\big\{\|\rr\| : \rr\in\Lambda\butnot\{\0\}\big\}.
\]
We can now state the following special case of the Dani--Kleinbock--Margulis correspondence principle:

\begin{proposition}[Corollary of {\cite[Theorem 8.5]{KleinbockMargulis}}]
\label{propositiondynamicalinterpretation}
For all $\bfA\in\MM$,
\begin{align*}
\omega(\bfA) &= \xi\big(\omega(\bfA;\mfa_+^*,\norm_0)\big),
\end{align*}
where
\begin{align*}
\norm_0\left(-\frac t\qdim,\cdots,-\frac t\qdim,\frac t\pdim,\cdots,\frac t\pdim\right) &= t,&
\xi(c) &= \frac \qdim\pdim \frac{1 + \pdim c}{1 - \qdim c}\cdot
\end{align*}
\end{proposition}
It is harder to state a multiplicative version of Proposition \ref{propositiondynamicalinterpretation}. In this context it is worth mentioning \cite[Theorem 9.2]{KleinbockMargulis}, which at first sight appears to be such a multiplicative analogue. However, the version of multiplicative approximation considered in \cite{KleinbockMargulis} differs in several senses from the version of multiplicative matrix approximation considered in this paper:
\begin{itemize}
\item the results there are for lattices rather than matrices, and the concepts become trivial when restricted to the ``usual example'' of lattices in the form $u_\bfA \Lambda_0$, as these lattices all contain vectors which lie in a coordinate subspace and are therefore $\psi$-MA for every $\psi$ in the sense of \cite[\69.1]{KleinbockMargulis}.
\item the ``height'' of a vector in a lattice is considered to be the maximum of its coordinates, whereas in our setup the height of the vector $(\pp,\qq)$ is considered to be the number $\prod_{j = 1}^\qdim |q_j|\vee 1$. This change (by itself) does not affect which matrices are considered to be VWMA, but it does affect the exponent of multiplicative irrationality for those matrices which are VWMA.
\end{itemize}
A multiplicative version of Proposition \ref{propositiondynamicalinterpretation} which is closer to our setup appeared in \cite[Proposition 3.1]{KMW}:
\begin{proposition}[Corollary of {\cite[Proposition 3.1]{KMW}}]
\label{propositiondynintmult}
A matrix $\bfA\in\MM$ is VWMA if and only if
\[
\omega(\bfA;\mfa_+,\norm) > 0,
\]
where $\norm:\mfa\to\Rplus$ is any norm.
\end{proposition}
This theorem does not contain any information relating the exponent of multiplicative irrationality function $\omega_\mult$ with functions of the form $\bfA\mapsto\omega(\bfA;\SS,\norm)$. This appears to be difficult or impossible to do for technical reasons.\Footnote{The integer point $\rr_1 = ((0,0),1) \in \R^{2 + 1}$ should be counted as a good multiplicative approximation of the matrix $\bfA = (1/2,\epsilon)^T \in \MM_{2,1}$ (since $u_\bfA \rr_1 = ((1/2,\epsilon),1)$ has a small second coordinate), but the point $\rr_2 = ((1,0),0)$ should not (since $u_\bfA \rr_2 = ((1,0),0)$ is independent of $\bfA$). But any $\tt\in\mfa$ which shrinks $u_\bfA \rr_1$ to a small size also shrinks $u_\bfA \rr_2$ to a small size. In \cite{KMW} this problem was circumvented by finding another approximant which can be shrunk to small size using only $\tt\in\mfa_+$, but this approximant may not be of as good quality as $\rr_1$. In some sense the real problem might be that the function $\Delta$ appearing in the definitions of $\omega$ and $\omega_\times$ does not give enough information as to how far a lattice is into the cusp.\label{footnotetechnical}}

\subsection{Computing the exponent of irrationality of an affine subspace of $\EE$}
In view of Propositions \ref{propositiondynamicalinterpretation} and \ref{propositiondynintmult}, and after replacing $\SS$ by a discrete approximation, to prove Theorem \ref{theoremexponents} it suffices to demonstrate the following:

\begin{theorem}
\label{theoremexponents2}
Let $\mu$ be a measure on $\MM$ which is supported on an affine subspace $\AA\subset\EE$ and which is weakly quasi-decaying relative to $\borel \subset \MM\cap\AA$ when interpreted as a measure on $\AA$. Fix $\SS\subset\mfa$ and $\norm:\SS\to\Rplus$ such that for all $\tt\in\SS$, we have $\norm(\tt) \asymp_\times \|\tt\|$. Then for $\mu$-a.e. $\bfA\in\MM\cap\AA$,
\begin{equation}
\label{exponents2}
\omega(\bfA;\SS,\norm) = \inf\{\omega(\bfB;\SS,\norm) : \bfB\in\MM\cap \AA\}.
\end{equation}
\end{theorem}

We now begin the preliminaries to the proof of this theorem, which involve finding an alternate expression for the right hand side of \eqref{exponents2}.

\begin{notation*}
Let $\VV$ denote the collection of all rational subspaces of $\R^{\pdim + \qdim}$. Note that $(\VV,\subset)$ is a partially ordered set whose maximal chains are all of length $(\pdim + \qdim)$. We will call the elements of $\VV$ ``vertices'', to emphasize that we are thinking about $\VV$ as a combinatorial object, namely a partially ordered set under inclusion. For each $V\in\VV$, $\bfA\in\MM$, and $\tt\in\mfa$, let
\[
f_{\tt,V}(\bfA) \equiv f_\tt(\bfA,V) \df \Covol\big(g_\tt u_\bfA(\Lambda_0\cap V)\big),
\]
where $\Covol$ denotes the covolume of a discrete subgroup of $\R^{\pdim + \qdim}$ with respect to some fixed norm on $\R^{\pdim + \qdim}$, relative to the $\R$-linear span of that discrete subgroup.
\end{notation*}

We will think of the number $f_{\tt,V}(\bfA)$ as a sort of ``accuracy of approximation'' of the rational subspace $V\leq \R^{\pdim + \qdim}$, relative to the window $\tt$, in analogy to how the number $\|g_\tt u_\bfA \rr\|$ can be thought of as the ``accuracy of approximation'' of an integer vector $\rr\in \Z^{\pdim+\qdim}$. The important thing is that smaller values of $f_{\tt,V}(\bfA)$ mean that $\bfA$ is more well approximable and larger values mean that it is less well approximable. The connection between the values of $f_{\tt,V}(\bfA)$ for various $\tt,V$ and the approximability of $\bfA$ in the sense of the Dani--Kleinbock--Margulis correspondence principle will be made more clear in the proof of Lemma \ref{lemmaeasydirection} below.

\begin{lemma}
\label{lemmaFtV}
For each $V\in\VV$ and $\tt\in\mfa$, there exists an affine map $F_{\tt,V}:\EE\to \bigwedge^{\dim(V)} \R^{\pdim + \qdim}$ such that for all $\bfA\in \MM$,
\[
f_{\tt,V}(\bfA) = \|F_{\tt,V}(\psi_{\bothdim}(\bfA))\|,
\]
where $\|\cdot\|$ is the wedge power of the norm used to define covolume.
\end{lemma}
\begin{proof}
Let $\bb_1,\ldots,\bb_v$ be an integral basis of $V$. Then the parallelepiped $\sum_{i = 1}^v [0,1]\bb_i$ is a fundamental domain for $\Lambda_0\cap V$. It follows that for each $\bfA\in\MM$, the parallelepiped $\sum_{i = 1}^v [0,1]g_\tt u_\bfA \bb_i$ is a fundamental domain for $g_\tt u_\bfA(\Lambda_0\cap V)$. Thus the covolume of $g_\tt u_\bfA(\Lambda_0\cap V)$ is equal to the volume of the parallelepiped, i.e.
\[
f_{\tt,V}(\bfA) = \Vol\left(\sum_{i = 1}^v [0,1]g_\tt u_\bfA \bb_i\right) = \|g_\tt u_\bfA \bb_1\wedge\cdots\wedge g_\tt u_\bfA \bb_v\|
= \|g_\tt u_\bfA(\bb_1\wedge\cdots\wedge\bb_v)\|.
\]
So to complete the proof, it suffices to show that for all $\tau\in \bigwedge^v \R^{\pdim + \qdim}$, the map
\[
F_\tau(\psi_{\bothdim}(\bfA)) \df u_\bfA(\tau)\in \bigwedge^v \R^{\pdim + \qdim}
\]
can be extended affinely to all of $\EE$. But this follows from the following explicit formula for $F_\tau$:
\[
F_\tau(\psi_{\bothdim}(\bfA))
= \sum_{\substack{I,J\subset\{1,\ldots,\pdim + \qdim\} \\ \#(I) = \#(J) = k \\ I\cap \{1,\ldots,\pdim\} \subset J}} \epsilon_{IJ} \tau_I \big[\psi_{\bothdim}(\bfA)\big]_{K(I,J)} \ee_J
\]
where we use the notations
\begin{align*}
\ee_I \df \bigwedge_{i\in I} &\ee_i, \;\;\;
\tau \df \sum_{\substack{I\subset \{1,\ldots,\pdim + \qdim\} \\ \#(I) = v}} \tau_I \ee_I,\;\;\;
\epsilon_{IJ}\in\{\pm 1\}, \;\;\; \big[\psi_{\bothdim}(\bfA)\big]_\smallemptyset \df 1,\\
K(I,J) &\df (\{1,\ldots,\pdim\}\cap (J\butnot I))\cup (\{\pdim + 1,\ldots,\pdim + \qdim\}\butnot (I\butnot J)).
\qedhere\end{align*}
\end{proof}

In the sequel we will extend $f_{\tt,V}$ to $\EE$ by letting $f_{\tt,V}(\sigma) = \|F_{\tt,V}(\sigma)\|$ for all $\sigma\in\EE$.

Given an affine subspace $\AA\subset\EE$, a set $\SS\subset\mfa$, and a function $\norm:\SS\to\Rplus$, let
\begin{equation}
\label{omegaLdef}
\omega(\AA;\SS,\norm) \df \limsup_{\SS\ni\tt\to\infty} \sup_{V\in\VV} \frac{-\log \|F_{\tt,V}\given \AA\|}{\norm(\tt) \dim(V)},
\end{equation}
where
\[
\|F\given\AA\| \df \|F(\zero_\AA)\| \vee \sup_{\substack{\sigma\in \AA \\ \|\sigma - \zero_\AA\| \leq 1}} \|F(\sigma) - F(\zero_\AA)\|.
\]
Here $\zero_\AA\in\AA$ is chosen so as to minimize $\|\zero_\AA\|$. We will show that $\omega(\AA;\SS,\norm)$ is equal to the right hand side of \eqref{exponents2}. One direction we can show now, and the other direction will follow from the proof of Theorem \ref{theoremexponents2}.

\begin{lemma}
\label{lemmaeasydirection}
With the above notation,
\[
\inf\big\{\omega(\bfA;\SS,\norm) : \bfA\in \MM\cap \AA\} \geq \omega(\AA;\SS,\norm).
\]
\end{lemma}
\begin{proof}
Fix $\bfA\in\MM\cap \AA$, $\tt\in\SS$, and $V\in\VV$. Then
\[
f_{\tt,V}(\bfA) \leq \|F_{\tt,V}\given \AA\|\cdot(1 + \|\psi_{\bothdim}(\bfA)\|).
\]
By Minkowski's theorem, there exists a vector $\vv\in \Lambda_0\cap V$ such that
\[
\|g_\tt u_\bfA \vv\| \leq 2f_{\tt,V}(\bfA)^{1/\dim(V)},
\]
so
\[
\Delta\big(g_\tt u_\bfA \Lambda_0\big) \geq \frac{-\log\big(2^{\dim(V)}(1 + \|\psi_{\bothdim}(\bfA)\|)\cdot \|F_{\tt,V}\given \AA\|\big)}{\dim(V)}\cdot
\]
Dividing by $\norm(\tt)$ and taking the limsup over $\SS\ni\tt\to\infty$ completes the proof.
\end{proof}

\subsection{Proof of Theorem \ref{theoremexponents2}}
By Lemma \ref{lemmaeasydirection}, to prove Theorem \ref{theoremexponents2} it suffices to show that for $\mu$-a.e. $\bfA\in\MM\cap \AA$, we have $\omega(\bfA;\SS,\norm) \leq \omega(\AA;\SS,\norm)$. We now state a lemma which will allow us to prove this:

\begin{lemma}
\label{lemmaexponents}
Let $\mu$ be a measure on $\MM$ which is supported on an affine subspace $\AA\subset\EE$ and which is uniformly weakly quasi-decaying and uniformly quasi-Federer relative to $\borel\subset \MM\cap\AA$ when interpreted as a measure on $\AA$. Let $X = \Supp(\mu) \subset \MM\cap\AA$. Fix $\gamma > 0$ and a ball $B_0 = B_X(\bfA_0,\rho_0)$. Consider $\tt\in\mfa$ and $0 < \kappa \leq 1$ such that
\begin{align}
\label{KLWassumption}
\sup_{2B_0} f_{\tt,V} \geq \kappa^{\dim(V)}
\end{align}
for all $V\in\VV$, and let
\[
W_{\kappa,\tt} \df \{\bfA\in X : \exists \vv\in\Lambda_0\butnot\{\0\} \;\; \|g_\tt u_\bfA \vv\| \leq e^{-\gamma\|\tt\|} \kappa\}.
\]
Then there exists $\epsilon > 0$ (depending on $\mu,\borel,\gamma$ but not $\kappa,\tt$) such that
\[
\mu(W_{\kappa,\tt}\cap B_0\cap \borel) \lesssim_\times e^{-\epsilon\|\tt\|}.
\]
\end{lemma}
In this lemma and its proof, we understand the metric on $X$ to be the one inherited from the vector space $\EE$, not the one inherited from the vector space $\MM$.
\begin{proof}[Proof of Theorem \ref{theoremexponents2} assuming Lemma \ref{lemmaexponents}]
Let $\borel\subset \MM\cap\AA$, $X = \Supp(\mu)$, $\gamma > 0$, and $B_0 \subset B_X(\bfA_0,\rho_0)$ be as in Lemma \ref{lemmaexponents}, with the additional constraint that $\borel\cap 2B_0\neq \emptyset$. Let $\tau = \omega(\AA;\SS,\norm) + \gamma$, and for each $\tt\in\SS$, let $\kappa_\tt = e^{-\tau \norm(\tt)}$. Fix $\bfA\in \borel\cap 2B_0$. Since $\mu$ is weakly quasi-decaying at $\bfA$ relative to $\borel$, it follows that $2B_0 = B_\AA(\bfA_0,2\rho_0)\cap \Supp(\mu)$ cannot be contained in an affine hyperplane of $\AA$, so the affine span of $2B_0$ is equal to $\AA$. Thus for all $V\in\VV$ we have
\[
\sup_{2B_0} f_{\tt,V} \asymp_\times \|F_{\tt,V}\given \AA\|
\]
and so by \eqref{omegaLdef}, if $\tt\in\SS$ is sufficiently large then
\[
\sup_{2B_0} f_{\tt,V} \geq e^{-\tau \norm(\tt) \dim(V)} = \kappa_\tt^{\dim(V)}.
\]
So by Lemma \ref{lemmaexponents}, if $\mu$ is uniformly weakly quasi-decaying relative to a set $\borel\subset X$ when interpreted as a measure on $\AA$, then
\[
\mu(W_{\kappa_\tt,\tt}\cap B_0\cap \borel) \lesssim_\times e^{-\epsilon\|\tt\|}.
\]
Thus the Borel--Cantelli lemma implies that for $\mu$-a.e. $\bfA\in B_0\cap \borel$ we have
\begin{equation}
\label{borelcantelli}
\#\{\tt\in\SS : \bfA\in W_{\kappa_\tt,\tt}\} < \infty.
\end{equation}
But if $\bfA$ satisfies \eqref{borelcantelli}, then
\begin{align*}
\omega(\bfA;\SS,\norm) &\leq \limsup_{\SS\ni\tt\to\infty} \frac{-\log(e^{-\gamma\|\tt\|} \kappa_\tt)}{\norm(\tt)}
= \tau + \gamma\limsup_{\SS\ni\tt\to\infty} \frac{\|\tt\|}{\norm(\tt)} \noreason\\
&= \omega(\AA;\SS,\norm) + (1 + C_1) \gamma. \note{for some $C_1 > 0$}
\end{align*}
Since $\gamma$ and $B_0$ were arbitrary, we have $\omega(\bfA;\SS,\norm)\leq \omega(\AA;\SS,\norm)$ for $\mu$-a.e. $\bfA\in \borel$. Combining with Lemmas \ref{lemmaunifQD}, \ref{lemmaquasifederer}, and \ref{lemmaeasydirection} completes the proof.
\end{proof}

Now we need to prove Lemma \ref{lemmaexponents}. The idea, following \cite{KleinbockMargulis,KLW}, is to construct a cover of the set $W_{\kappa,\tt}\cap B_0$ whose measure can be bounded using fact that $\mu$ is uniformly weakly quasi-decaying relative to $\borel$. To construct this cover, we will first construct a tree $\TT$ such that each node $e\in\TT$ corresponds to a ball in $B_e\subset \AA$. We will also associate to $e$ a \emph{flag}, i.e. a set $\FF_e \equiv \{V_0,\ldots,V_\ell\} \subset \VV$ such that $\{\0\} = V_0 \propersubset V_1 \propersubset \cdots \propersubset V_\ell = \R^{\pdim + \qdim}$. The purpose of the flag $\FF_e$ is to separate potential approximants to points in $B_e$ into $\ell$ different classes: an approximant $\rr\in\Lambda_0\butnot\{\0\}$ is in exactly one of the sets $V_1\butnot V_0,\ldots,V_\ell\butnot V_{\ell - 1}$. In order for this separation to be useful, the flag $\FF_e$ should satisfy the following conditions:
\begin{itemize}
\item[(1)] The quality-of-approximation ratios $f_\tt(B_e,V_{i + 1})/f_\tt(B_e,V_i)$ should not be too large (in terms of $\omega(\AA;\SS,\norm)$), where
\begin{equation}
\label{ftSV}
f_\tt(S,V) \df \sup_S f_{\tt,V}.
\end{equation}
\label{page21}
\item[(2)] For each vertex $V\in\VV$ such that $V_i\subsetneqq V\subsetneqq V_{i + 1}$ for some $i$, the quality of approximation $f_\tt(B_e,V)$ should be bounded from below in terms of $f_\tt(B_e,V_i)$ and $f_\tt(B_e,V_{i + 1})$.
\end{itemize}
The idea of the tree is to give us a picture of what happens as we ``zoom in'' towards a point $\bfA\in X$. On the large scale, we will be able to find a flag $\FF$ which satisfies (1) and (2) which depends only on the Diophantine properties of the affine space $\AA$. As we zoom in, all vertices become better approximations (because the supremum in \eqref{ftSV} is taken over a smaller collection). If this causes a vertex to become a counterexample to (2), then we simply add it to our flag and create a new node on the tree. On the other hand, the probability that a vertex will become a counterexample to (1) (assuming that when we added the vertex to the flag, it satisfied (1)) is small because of the quasi-decay condition. So if $\bfA$ is a typical point, then after we finish the zooming process the flag will still satisfy (1). Lemma \ref{lemmaexistsV} below shows that in this case, $\bfA$ cannot be in $W_{\kappa,\tt}$.

We will encode the Diophantine properties of the root flag $\FF_\smallemptyset$ by defining a function $\eta: \{0,\ldots,\pdim + \qdim\}\to (0,\infty)$ such that for each $j$, $\eta(j)$ represents the quality of the ``best expected approximation'' in dimension $j$. As we zoom in, we will add the vertex $V$ to our flag at the exact moment when the quality of approximation of $V$ becomes better than $\eta(\dim(V))$. For a typical point, this strategy should create a final flag which satisfies condition 1.

\begin{figure}[h!]
\begin{center}
\begin{tabular}{@{}ll|@{}}
\begin{tikzpicture}[line cap=round,line join=round,>=triangle 45,scale=0.6]
\clip(-1.0744866397712483,-1.0652581632465994) rectangle (12.52932974604722,6.807415187466945);
\draw[->] (1.2,2.0)-- (11.4,2.0);
\draw (1.6296536816006593,1.9962393331765842)-- (2.69,0.42);
\draw (4.09,-0.39)-- (2.69,0.42);
\draw (4.09,-0.39)-- (6.89,0.16);
\draw (6.89,0.16)-- (9.694690100098676,1.9776280691815895);
\draw[->] (1.62,1.2)-- (1.62,5.950803061223502);
\draw[thick, dotted] (8.29,2.55)-- (8.29,6.54);
\draw[thick, dotted] (5.49,1.11)-- (5.49,6.11);
\begin{scriptsize}
\draw[color=black] (11.4,1.62) node {dimension};
\draw[color=black] (0.6,5.4) node {height};
\draw[color=black] (0.6,5) node {(log scale)};
\draw[color=black] (1.0,2.0) node {$0$};
\draw [fill=black] (1.62,1.99) circle (1.5pt);
\draw [fill=black] (2.69,0.42) circle (1.5pt);
\draw [fill=black] (4.09,-0.39) circle (1.5pt);
\draw [fill=black] (6.89,0.16) circle (1.5pt);
\draw [fill=black] (9.69,2.0) circle (1.5pt);
\draw [fill=black] (5.49,1.11) circle (1.5pt);
\draw [fill=black] (8.29,2.54) circle (1.5pt);
\end{scriptsize}
\end{tikzpicture}
\begin{tikzpicture}[line cap=round,line join=round,>=triangle 45,scale=0.6]
\clip(-1.0744866397712483,-1.0652581632465994) rectangle (12.52932974604722,6.807415187466945);
\draw[->] (1.2,2.0)-- (11.4,2.0);
\draw (1.6296536816006593,1.9962393331765842)-- (2.69,0.42);
\draw (4.09,-0.39)-- (2.69,0.42);
\draw (4.09,-0.39)-- (6.89,0.16);
\draw (6.89,0.16)-- (9.694690100098676,1.9776280691815895);
\draw[->] (1.62,1.2)-- (1.62,5.950803061223502);
\draw[thick, dotted] (8.29,1.55)-- (8.29,5.54);
\begin{scriptsize}
\draw[color=black] (11.4,1.62) node {dimension};
\draw[color=black] (0.6,5.4) node {height};
\draw[color=black] (0.6,5) node {(log scale)};
\draw[color=black] (1.0,2.0) node {$0$};
\draw [fill=black] (1.62,1.99) circle (1.5pt);
\draw [fill=black] (2.69,0.42) circle (1.5pt);
\draw [fill=black] (4.09,-0.39) circle (1.5pt);
\draw [fill=black] (6.89,0.16) circle (1.5pt);
\draw [fill=black] (9.69,2.0) circle (1.5pt);
\draw [fill=black] (5.49,-0.11) circle (1.5pt);
\draw [fill=black] (8.29,1.55) circle (1.5pt);
\end{scriptsize}
\end{tikzpicture}
\end{tabular}
\caption{Two possible plots of the set\[\hspace{-1 in}\{(\dim(V),\log f_\tt(B(\bfA,\rho),V)) : \text{$V\in\VV$ is $\FF$-addable or satisfies $V\in\FF$}\},\]
along with a graph of the piecewise linear function $\log\eta$. The displayed points represent the minimum plot points over each vertical strip, and the vertical ellipses represent additional ungraphed plot points. The two plots are taken with the same value of $\tt$ and $\bfA$ but different values for $\rho$. As $\rho$ decreases, all plot points will move down, but the probability that any given plot point jumps down a significant amount is small (under the assumption that $\bfA$ is $\mu$-random). Once a plot point ``crosses'' the graph of $\log\eta$, then its corresponding vertex is added to the flag, at which point it is unlikely to move down further. This explains why in the typical case the final plot is essentially the same as the graph of $\log\eta$ on the integers.}
\label{figureflag}
\end{center}
\end{figure}

\begin{definition*}
If $\FF\subset\VV$ is a flag, then the number $\ell \equiv \ell(\FF) \df \#(\FF) - 1$ 
is called the \emph{length} of the flag. A vertex $V\in\VV\butnot\FF$ is \emph{$\FF$-addable} if $\FF\cup\{V\}$ is a flag. A flag $\FF$ is called \emph{maximal} if $\ell(\FF) = \pdim + \qdim$, or equivalently if there is no $\FF$-addable vertex.
\end{definition*}

\begin{definition*}
Given $\eta:\{0,\ldots,\pdim + \qdim\}\to (0,\infty)$ and a vertex $V\in\VV$, a set $S\subset\EE$ is said to be \emph{$(\eta,V)$-approximable} if
\[
f_\tt(S,V) \leq \eta(\dim(V)).
\]
If $S\subset\EE$ is fixed, then the collection of vertices $V\in\VV$ such that $S$ is $(\eta,V)$-approximable will be denoted $\WW(\eta,S)$, and its complement will be denoted $\BB(\eta,S)$.
\end{definition*}

\begin{lemma}[Cf. {\cite[Proposition 5.1]{KLW}}]
\label{lemmaexistsV}
Let $\FF\subset \WW(\eta,S)$ be a maximal flag, and fix $\bfA\in W_{\kappa,\tt}\cap S$. Then there exists $V\in\FF\butnot\{\0\}$ such that
\[
f_{\tt,V}(\bfA) \leq e^{-\gamma\|\tt\|} \kappa \eta(\dim(V) - 1).
\]
\end{lemma}
\begin{proof}
Since $\bfA\in W_{\kappa,\tt}$, we have $\|g_\tt u_\bfA \vv\| \leq e^{-\gamma\|\tt\|} \kappa$ for some $\vv\in\Lambda_0$. Write $\FF = \{V_0,\ldots,V_{\pdim + \qdim}\}$ with $\{\0\} = V_0 \propersubset V_1 \propersubset \cdots \propersubset V_{\pdim + \qdim} = \R^{\pdim + \qdim}$. Let $i$ be the largest element of $\{0,\ldots,\pdim + \qdim\}$ such that $\vv\notin V_i$. Then $V_{i + 1} = V_i + \R\vv$. An argument based on the geometric significance of $f_\tt$ shows that
\[
f_\tt(\bfA,V_{i + 1}) \leq \|g_\tt u_\bfA \vv\| f_\tt(\bfA,V_i).
\]
On the other hand, since $\FF\subset \WW(\eta,S)$ and $\bfA\in S$ we have
\[
f_\tt(\bfA,V_i) \leq \eta(\dim(V_i)) = \eta(\dim(V_{i + 1}) - 1)
\]
and by the definition of $\vv$,
\[
\|g_\tt u_\bfA \vv\| \leq e^{-\gamma\|\tt\|} \kappa.
\]
Combining these inequalities completes the proof.
\end{proof}

\begin{definition*}
Fix $\lambda\geq 2$, a flag $\FF\subset\VV$, and a function $\eta:\{0,\ldots,\pdim + \qdim\}\to (0,\infty)$. A ball $B = B_X(\bfA,\rho)$ is said to be \emph{$(\FF,\eta,\lambda)$-permissible} if $\FF\subset \WW(2\eta,2B)$, but every $\FF$-addable vertex is in $\BB(\eta,\lambda B)$.
\end{definition*}

\begin{definition*}
Fix a flag $\FF\subset\VV$ and a function $\eta:\{0,\ldots,\pdim + \qdim\}\to (0,\infty)$. We say that $\eta$ is \emph{$\FF$-concave} if for all $j\notin \{\dim(V):V\in\FF\}$,
\[
\eta(j) \geq 8\sqrt{\eta(j - 1)\eta(j + 1)}.
\]
\end{definition*}

The purpose of concavity is to ensure that if $B$ is $(\FF,\eta,\lambda)$-permissible, then the flag $\FF$ will satisfy condition (2) on p.\pageref{page21}. The factor of $8$ will be important in the proof of \eqref{rx3t} below.

\begin{remark}
\label{remarketa}
If $\eta$ is $\FF$-concave and $\FF = \{V_0,\ldots,V_\ell\}$ with $V_0\propersubset \cdots \propersubset V_\ell$, then for each $i = 0,\ldots,\ell - 1$, if $j = \dim(V_i)$ and $m = \dim(V_{i + 1}) - \dim(V_i)$, then for each $0\leq k \leq m$ we have
\[
\eta(j + k) \geq 8^{k(m - k)} \eta(j + m)^{k/m} \eta(j)^{(m - k)/m}.
\]
\end{remark}

\begin{notation*}
For each $i = 0,\ldots,\pdim + \qdim$ let
\[
C_i \df 4^{i(\pdim + \qdim - i)}.
\]
Note that $C_0 = C_{\pdim + \qdim} = 1$, and $C_i = 4\sqrt{C_{i - 1} C_{i + 1}}$ for all $0 < i < \pdim + \qdim$.
\end{notation*}

\begin{lemma}[Base case]
\label{lemmabasecase}
Fix a ball $B_0 = B_X(\bfA_0,\rho_0)$. Then there exist a flag $\FF_0\subset\VV$ and an $\FF_0$-concave function $\eta:\{0,\ldots,\pdim + \qdim\}\to (0,\infty)$ such that:
\begin{itemize}
\item[(i)] $B_0$ is $(\FF_0,\eta,2)$-permissible for every $V\in\FF_0$,
\item[(ii)] $\FF_0\subset \BB(\eta,2B_0)$,
\item[(iii)] $\eta(j) \leq C_j/2 \all j$, and
\item[(iv)] $\eta(j + 1)/\eta(j) \gtrsim_\times \kappa \all j$.
\end{itemize}
\end{lemma}
\begin{proof}
For each $V\in\VV$ let
\begin{align*}
f(V) &\df f_\tt(2B_0,V),&
g(V) &\df \frac{\log(f(V)/C_{\dim(V)})}{\dim(V)},
\end{align*}
with the convention that $g(\{\0\}) = -\infty$. Note that by \eqref{KLWassumption}, $g(V)\gtrsim_\plus \log(\kappa)$ for all $V\propersupset \{\0\}$. Let $V_0 = \{\0\}$, and recursively define $V_1,\ldots,V_\ell$ by letting $V_{i + 1}\propersupset V_i$ satisfy
\[
g(V_{i + 1}) = \min\{g(V): V\propersupset V_i\}.
\]
This process halts when $V_\ell = \R^{\pdim + \qdim}$. Note that
\[
g(V_0) \leq g(V_1) \leq \cdots \leq g(V_\ell) = 0.
\]
Let $\FF_0 = \{V_0,\ldots,V_\ell\}$, and for each $i = 0,\ldots,\ell$ let $j_i = \dim(V_i)$ and
\begin{equation}
\label{etadef}
\eta(j_i) = f(V_i)/2.
\end{equation}
Extend $\eta$ to a map $\eta:\{0,\ldots,\pdim + \qdim\}\to(0,\infty)$ which is minimal subject to being $\FF_0$-concave. Equivalently, this extension can be described by the requirement that for each $i = 0,\ldots,\ell - 1$, the function
\[
\theta(j) = \log(2\eta(j)/C_j)
\]
is linear on $\{j_i,\ldots,j_{i + 1}\}$.

For all $i$, since $g(V_i)\leq 0$, we have $\theta(j_i)= j_i g(V_i) \leq 0$. So by linearity, we have $\theta(j)\leq 0$ for all $j$, i.e. $\eta(j)\leq C_j/2$. This demonstrates (iii).

By \eqref{etadef}, we have $\FF_0\subset\BB(\eta,2B_0) \cap \WW(2\eta,2B_0)$. This demonstrates (ii) and the first part of (i). To demonstrate the second part of (i), suppose $V$ is an $\FF$-addable vertex, and write $V_i\propersubset V\propersubset V_{i + 1}$ for some $i = 0,\ldots,\ell - 1$. By the definitions of $V_i$ and $V_{i + 1}$, we have
\[
g(V_i) \leq g(V_{i + 1}) \leq g(V)
\]
and thus
\begin{align*}
\theta(j_i) &\leq j_i g(V),&
\theta(j_{i + 1}) &\leq j_{i + 1} g(V).
\end{align*}
Since $\theta$ is linear on $\{j_i,\ldots,j_{i + 1}\}$, this implies that
\[
\theta(j) \leq j g(V) = \log(f(V)/C_j).
\]
Rearranging gives $\eta(j) \leq f(V)/2$, so $V\in\BB(\eta,2B_0)$ and thus $B_0$ is $(\FF_0,\eta,2)$-permissible.

Finally, to demonstrate (iv), we note that since
\[
\frac{\theta(j_i)}{j_i} = g(V_i) \leq g(V_{i + 1}) = \frac{\theta(j_{i + 1})}{j_{i + 1}},
\]
we have
\[
\frac{\theta(j_{i + 1}) - \theta(j_i)}{j_{i + 1} - j_i} \geq \frac{\theta(j_{i + 1}) - j_i\frac{\theta(j_{i + 1})}{j_{i + 1}}}{j_{i + 1} - j_i} = \frac{\theta(j_{i + 1})}{j_{i + 1}} = g(V_{i + 1}) \gtrsim_\plus \log(\kappa).
\]
By the piecewise linearity of $\theta$, we have $\theta(j + 1) - \theta(j) \gtrsim_\plus \log(\kappa)$ for all $j$, and writing this inequality in terms of $\eta$ yields (iv).
\end{proof}

\begin{lemma}[Inductive step]
\label{lemmainductivestep}
Fix $\lambda \geq 2$, a non-maximal flag $\FF\subset\VV$, an $\FF$-concave function $\eta:\{0,\ldots,\pdim + \qdim\}\to (0,\infty)$, and an $(\FF,\eta,\lambda)$-permissible ball $B = B_X(\bfA_0,\rho_0) \subset \AA$. Then for each $\bfA\in B$, there exists an $\FF$-addable vertex $V_\bfA$ and an $(\FF\cup\{V_\bfA\},\eta,8\lambda)$-permissible ball $B_\bfA = B_X(\bfA,\rho_\bfA)$ such that
\begin{align} \label{Vpermissible}
V_\bfA &\in\BB(\eta, 8\lambda B_\bfA) \\ \label{containment}
2B_\bfA &\subset 2B \\ \label{rx3t}
8\lambda \rho_\bfA &\geq 2^{-(\pdim + \qdim)} e^{-2\|\tt\|_\infty}.
\end{align}
\end{lemma}
\begin{remark*}
The condition \eqref{rx3t} is the key ``new'' element of the proof of Theorem \ref{theoremexponents2} which has no analogue in \cite{KleinbockMargulis2, KLW, BKM, KMW}; it will allow us to prove the bound $\beta \leq \rho^\gamma$ for the hyperplane-neighborhoods whose $\mu$-measures we want to bound, thus allowing the weak quasi-decay condition to be used as a substitute for friendliness.
\end{remark*}
\begin{proof}
For each $\FF$-addable vertex $V$ let $\rho_{\bfA,V}$ be the smallest value $\rho\in 2^\Z$ such that
\begin{equation}
\label{rhoAVdef}
V\in \BB(\eta,B_X(\bfA,8\lambda \rho)),
\end{equation}
with $\rho_{\bfA,V} = 0$ if \eqref{rhoAVdef} holds for all $\rho\in 2^\Z$. Let
\[
\rho_\bfA \df \max\{\rho_{\bfA,V}:\text{$V\in\VV$ is $\FF$-addable}\},
\]
let $V_\bfA$ be an $\FF$-addable vertex such that $\rho_\bfA = \rho_{\bfA,V_\bfA}$, and let $B_\bfA = B_X(\bfA,\rho_\bfA)$. Let $\FF_\bfA = \FF\cup\{V_\bfA\}$. By construction, the set $\BB(\eta,8\lambda B_\bfA)$ contains every $\FF$-addable vertex. In particular, \eqref{Vpermissible} holds.

On the other hand, the $(\FF,\eta,\lambda)$-permissibility of $B$ and the minimality of $\rho_{\bfA,V_\bfA}$ together imply that
\[
f_\tt\big(\lambda B,V_\bfA\big) > \eta(\dim(V_\bfA)) \geq f_\tt\big(4\lambda B_\bfA,V_\bfA\big);
\]
it follows that
\[
B_X(\bfA_0,\lambda \rho_0) \nsubset B_X(\bfA,4\lambda \rho_\bfA).
\]
Since $\dist_X(\bfA_0,\bfA) \leq \rho_0$, this implies
\[
4\lambda \rho_\bfA < (\lambda + 1) \rho_0 < 2\lambda \rho_0,
\]
so $\rho_\bfA < \rho_0/2$ and thus \eqref{containment} holds. In particular $\FF\subset \WW(\eta,2B_\bfA)$. On the other hand, $V_\bfA\in \WW(\eta,2B_\bfA)$ since $\rho_\bfA = \rho_{\bfA,V_\bfA}$. Moreover, as noted in the previous paragraph the set $\BB(\eta,8\lambda B_\bfA)$ contains every $\FF$-addable vertex and in particular every $\FF_\bfA$-addable vertex. Thus $B_\bfA$ is $(\FF_\bfA,\eta,8\lambda)$-permissible.

To demonstrate \eqref{rx3t}, we will find an $\FF$-addable vertex $V$ such that $4\lambda \rho_{\bfA,V} \geq  2^{-(\pdim + \qdim)} e^{-2\|\tt\|_\infty}$. Write $\FF = \{V_0,\ldots,V_\ell\}$ with $\{\0\} = V_0 \propersubset V_1 \propersubset \cdots \propersubset V_\ell = \R^{\pdim + \qdim}$. Since $\FF$ is not maximal, we have $m \df \dim(V_{i + 1}) - \dim(V_i) \geq 2$ for some $i$. Let $W_i \df g_\tt u_\bfA(V_i)$, $W_{i + 1} \df g_\tt u_\bfA(V_{i + 1})$, and $\Lambda \df g_\tt u_\bfA \Lambda_0$. Applying Minkowski's theorem to the vector space $W_{i + 1}/W_i$ with the lattice $(\Lambda\cap W_{i + 1})/W_i$, we see that there exists a vector $\ww = g_\tt u_\bfA \vv \in \Lambda\cap W_{i + 1}\butnot W_i$ such that
\[
\dist(\ww,W_i) \leq 2\Covol\big((\Lambda\cap W_{i + 1})/W_i\big)^{1/m}.
\]
Let $V \df V_i + \R\vv$ and $W \df g_\tt u_\bfA(V) = W_i + \R\ww$. Note that $V$ is an $\FF$-addable vertex. We have
\[
\Covol(\Lambda\cap W) = \dist(\ww,W_i)\Covol(\Lambda\cap W_i)
\]
and
\[
\Covol\big((\Lambda\cap W_{i + 1})/W_i\big)
= \frac{\Covol(\Lambda\cap W_{i + 1})}{\Covol(\Lambda\cap W_i)}\cdot
\]
It follows that
\begin{align*}
f_{\tt,V}(\bfA) &= \Covol(\Lambda\cap W)\\
&\leq 2\Covol(\Lambda\cap W_{i + 1})^{1/m}\Covol(\Lambda\cap W_i)^{(m - 1)/m}\\
&= 2f_\tt(\bfA,V_{i + 1})^{1/m} f_\tt(\bfA,V_i)^{(m - 1)/m}.
\end{align*}
Let $j = \dim(V_i)$. Since $\bfA$ is $(\FF,\eta)$-permissible, we have
\begin{equation}
\label{xVbound}
f_{\tt,V}(\bfA) \leq 4\eta(j + m)^{1/m} \eta(j)^{(m - 1)/m} \leq \frac{1}{2}\eta(j + 1)
\end{equation}
where the last inequality follows from Remark \ref{remarketa}. Let $\epsilon = 2^{-(\pdim + \qdim)}$, and note that $(1 + \epsilon)^{\pdim + \qdim} \leq 2$. For all $\bfB\in B_X(\bfA,\epsilon e^{-2\|\tt\|_\infty})$,
\begin{align*}
f_\tt(\bfB,V) &\leq \|g_\tt u_\bfB (g_\tt u_\bfA)^{-1}\|^{\dim(V)} f_{\tt,V}(\bfA)\\
&= \big\|u\big(\diag(e^{t_1},\ldots,e^{t_\pdim})(\bfB - \bfA)\diag(e^{-t_{\pdim + 1}},\ldots,e^{-t_{\pdim + \qdim}})\big)\big\|^{j + 1} f_{\tt,V}(\bfA)\\
&\leq (1 + e^{2\|\tt\|_\infty}\|\bfB - \bfA\|)^{j + 1} \frac12\eta(j + 1)\\
&\leq (1 + \epsilon)^{\pdim + \qdim} \frac12\eta(j + 1) = \eta(j + 1) = \eta(\dim(V)).
\end{align*}
Thus by definition, $8\lambda \rho_{\bfA,V} \geq \epsilon e^{-2\|\tt\|_\infty}$.
\end{proof}

For each $i = 0,\ldots,\pdim + \qdim$ write $\lambda_i = 2\cdot 8^i$. Let $B_0,\FF_0,\eta$ be as in Lemma \ref{lemmabasecase}, and let
\[
\PP \df \{(B,\FF) : \text{$B$ is $(\FF,\eta,\lambda_{\ell(\FF)})$-permissible, $B\subset B_0$, and $\FF\supset\FF_0$}\},
\]
so that $(B_0,\FF_0)\in \PP$. We will now construct a tree in $\PP$ with $(B_0,\FF_0)$ as the root node.

{\bf Construction of children.} Fix $(B,\FF)\in\PP$, and let $\lambda = \lambda_{\ell(\FF)}$. Since $\FF\supset\FF_0$, $\eta$ is $\FF$-concave, so Lemma \ref{lemmainductivestep} applies. For each $\bfA\in B\cap \borel$ let $\rho_\bfA > 0$ be as in Lemma \ref{lemmainductivestep}, so that $\{B_\bfA = B_X(\bfA,\rho_\bfA) : \bfA\in B\cap \borel\}$ is a cover of $B\cap \borel$. By the $4r$-covering lemma (see e.g. \cite[Theorem 8.1]{MSU}), there exists a finite set $(\bfA_i)_{i = 1}^\Index$ such that the collection $\{B_i = B_{\bfA_i} : i = 1,\ldots,\Index\}$ still covers $B\cap \borel$, but the collection $\{(1/4)B_i : i = 1,\ldots,\Index\}$ is disjoint. For each $i$, let $\FF_i = \FF\cup\{V_{\bfA_i}\}$, so that $(B_i,\FF_i)\in \PP$. Let
\[
\CC(B,\FF) \df \{(B_i,\FF_i) : i = 1,\ldots,\Index\} \subset \PP,
\]
and note that
\begin{equation}
\label{children}
B\subset \bigcup\{B_i : (B_i,\FF_i)\in \CC(B,\FF)\}.
\end{equation}
{\bf Covering argument.}
Let $\ell_0 = \ell(\FF_0)$, let $\TT_{\ell_0} \df \{(B_0,\FF_0)\}$, and for each $i = \ell_0 + 1,\ldots,\pdim + \qdim$ let
\[
\TT_i ~\df \bigcup_{(B,\FF)\in \TT_{i - 1}} \CC(B,\FF).
\]
Fix $\bfA\in B_0\cap \borel$. By \eqref{children}, can recursively define a sequence $(B_i,\FF_i)_{i = \ell_0}^{\pdim + \qdim}$ such that for each $i = \ell_0,\ldots,\pdim + \qdim$, we have $(B_i,\FF_i)\in\TT_i$, $\bfA\in B_i$, and if $i > \ell_0$, then
\[
(B_i,\FF_i) \in \CC(B_{i - 1},\FF_{i - 1}).
\]
Write $\FF_i = \FF_{i - 1}\cup\{V_i\}$, so that by Lemma \ref{lemmainductivestep}, $V_i\in \BB(\eta,\lambda_i B_i)$. Also write $\FF_0 = \FF_{\ell_0} = \{V_0,\ldots,V_{\ell_0}\}$, so that by Lemma \ref{lemmabasecase}, $V_i \in \BB(\eta,2 B_0) \subset \BB(\eta,\lambda_{\ell_0} B_{\ell_0})$ for all $i = 0,\ldots,\ell_0$.

If $\bfA\in W_{\kappa,\tt}$, then by Lemma \ref{lemmaexistsV} there exists $i = 0,\ldots,\pdim + \qdim$ such that 
\[
f_\tt(\bfA,V_i) \leq e^{-\gamma\|\tt\|} \kappa \eta(\dim(V_i) - 1).
\] 
Combining with part (iv) of Lemma \ref{lemmabasecase} gives
\[
f_\tt(\bfA,V_i) \lesssim_\times e^{-\gamma\|\tt\|} \eta(\dim(V_i)).
\]
To summarize,
\[
W_{\kappa,\tt}\cap B_0 \cap \borel \subset \bigcup_{i = \ell_0}^{\pdim + \qdim} \bigcup_{(B,\FF)\in\TT_i} \bigcup_{\substack{V\in\FF \\ V\in \BB(\eta,\lambda_i B_i)}} (W_{\kappa,\tt}(V)\cap B),
\]
where
\[
W_{\kappa,\tt}(V) \df \{\bfA\in B_0 : f_{\tt,V}(\bfA) \leq C e^{-\gamma\|\tt\|} \eta(\dim(V))\}
\]
for some $C > 0$.
\begin{claim}
Fix a ball $B \subset X$ and $V\in\BB(\eta,\lambda_{\pdim + \qdim} B)$. Then
\[
\mu(W_{\kappa,\tt}(V)\cap B\cap \borel) \lesssim_\times e^{-\alpha\|\tt\|} \mu(\lambda_{\pdim + \qdim} B)
\]
for some $\alpha > 0$ depending only on $\mu,\borel,\gamma$.
\end{claim}
\begin{subproof}
Since $V\in\BB(\eta,\lambda_{\pdim + \qdim} B)$, there exists $\bfA\in\lambda_{\pdim + \qdim} B$ such that $f_{\tt,V}(\bfA) > \eta(\dim(V))$. Let $F_{\tt,V}:\EE \to \EE_V \df \bigwedge^{\dim(V)}\R^{\pdim + \qdim}$ be as in Lemma \ref{lemmaFtV}, and let $\pi:\EE_V\to\R$ be a linear map such that
\[
|\pi\circ F_{\tt,V}(\bfA)| \asymp_\times \|F_{\tt,V}(\bfA)\| \text{ and } \|\pi\| = 1.
\]
Let $\LL = (\pi\circ F_{\tt,V})^{-1}(0)\in\Hyp(\AA)$. Then there exists $c > 0$ (depending on $B,V$) such that for all $\bfB\in\MM\cap\AA$,
\[
\dist(\bfB,\LL) = c |\pi\circ F_{\tt,V}(\bfB)|.
\]
Then
\[
\|d_\LL\|_{\mu,\lambda_{\pdim + \qdim}B} \geq \dist(\bfA,\LL) = c |\pi\circ F_{\tt,V}(\bfA)| \asymp_\times c f_{\tt,V}(\bfA) > c\eta(\dim(V)).
\]
So for all $\bfB\in W_{\kappa,\tt}(V)$, we have
\[
\dist(\bfB,\LL) \leq c f_{\tt,V}(\bfB) \lesssim_\times e^{-\gamma\|\tt\|} c\eta(\dim(V)) \lesssim_\times e^{-\gamma\|\tt\|} \|d_\LL\|_{\mu,\lambda_{\pdim + \qdim}B}.
\]
Letting $C$ denote the implied constant, we have
\[
W_{\kappa,\tt} \subset \thickvar\LL{C e^{-\gamma\|\tt\|} \|d_\LL\|_{\mu,\lambda_{\pdim + \qdim} B}}.
\]
Let $\beta \df C e^{-\gamma\|\tt\|}$, and let $\rho$ be the radius of $B$. By \eqref{rx3t},
\begin{equation}
\label{r3t}
8\lambda_{\pdim + \qdim}\rho \geq 2^{-(\pdim + \qdim)} e^{-2\|\tt\|_\infty}
\end{equation}
and thus $\beta \lesssim_\times \rho^{\gamma/2}$. Letting $\alpha \df \alpha(\gamma/2,\mu) > 0$ (cf. Lemma \ref{lemmaQD}), we have
\[
\mu\big(W_{\kappa,\tt}(V)\cap B\cap \borel\big)
\leq \mu\big(\thickvar\LL{\beta \|d_\LL\|_{\mu,\lambda_{\pdim + \qdim}B}}\cap \lambda_{\pdim + \qdim} B\cap \borel\big)
\lesssim_\times e^{-\alpha\|\tt\|} \mu(\lambda_{\pdim + \qdim} B).
\qedhere\]
\end{subproof}

So we get
\[
\mu(W_{\kappa,\tt}\cap B_0\cap \borel) \lesssim_\times e^{-\alpha\|\tt\|} \sum_{i = \ell_0}^{\pdim + \qdim} \sum_{(B,\FF)\in\TT_i} \mu(\lambda_{\pdim + \qdim} B).
\]
Let $\epsilon = \alpha/(\pdim + \qdim + 1) > 0$. To complete the proof of Lemma \ref{lemmaexponents}, it suffices to show that for all $i = \ell_0,\ldots,\pdim + \qdim - 1$, we have
\begin{equation}
\label{indhyp}
\sum_{(B,\FF)\in\TT_i} \mu(\lambda_{\pdim + \qdim} B) \lesssim_\times e^{i\epsilon\|\tt\|}.
\end{equation}
We prove \eqref{indhyp} by induction on $i$. When $i = \ell_0$, it holds trivially since $B_0$ is fixed. If it holds for $i$, then
\begin{align*}
\sum_{(B,\FF)\in\TT_{i + 1}} \mu\big(\lambda_{\pdim + \qdim} B\big)
&\lesssim_\times e^{\epsilon\|\tt\|}\sum_{(B,\FF)\in\TT_{i + 1}}\mu\big((1/4) B\big) \by{\eqref{quasifederer} and \eqref{r3t}}\\
&=_\pt e^{\epsilon\|\tt\|} \sum_{(B,\FF)\in \TT_i} \sum_{(B',\FF')\in \CC(B,\FF)} \mu\big((1/4) B'\big) \noreason\\
&\leq_\pt e^{\epsilon\|\tt\|}\sum_{(B,\FF)\in\TT_i} \mu\big(\lambda_{\pdim + \qdim} B\big) \note{disjointness}\\
&\lesssim_\times e^{(i + 1)\epsilon\|\tt\|}, \by{\eqref{indhyp}}
\end{align*}
i.e. \eqref{indhyp} holds for $i + 1$. This completes the proof of Lemma \ref{lemmaexponents} and thus of Theorems \ref{theoremexponents2}, \ref{theoremexponents}, and \ref{theoremKLW}.

\draftnewpage
\appendix
\section{Counterexample to a hypothesis of KLW}
\label{appendix1}

In \cite[para. after Theorem 6.1]{KLW}, KLW state that ``the non-uniform Federer condition [is] measure class invariant, and it is plausible that the same holds for the non-uniform decay condition.'' The following theorem shows on the contrary that the non-uniform decay condition is not measure class invariant:

\begin{theorem}
\label{theoremmeasureclass}
There exists a measure $\mu$ on $\R$ in the same measure class as Lebesgue measure such that for all $C,\alpha,\rho_0 > 0$ and $x\in\R$, there exist $0 < \rho \leq \rho_0$, $y\in\R$, and $0 < \beta \leq 1$ such that
\begin{equation}
\label{measureclass}
\mu\big(B(y,\beta\rho)\cap B(x,\rho)\big) > C \beta^\alpha \mu\big(B(x,\rho)\big).
\end{equation}
In particular, $\mu$ is not non-uniformly decaying in the sense of \cite[\66]{KLW}; thus non-uniform decay is not a measure class invariant.
\end{theorem}
\begin{proof}
Let $(q_\Index)_{\Index\in\N}$ be a dense sequence in $\R$, let $a_\Index = 2^{-\Index}$, let $b_\Index = 2^{2^\Index}$, and let
\[
f_\Index = a_\Index b_\Index \one_{B(q_\Index,1/b_\Index)},
\]
where $\one_S$ denotes the characteristic function of a set $S$. Then $\|f_\Index\|_1 = 2 a_\Index$, so $f \df \sum_{\Index = 1}^\infty f_\Index \in L^1(\R)$. Let $\mu = (1 + f)\lambda$, where $\lambda$ denotes Lebesgue measure. Fix $x\in\R$ and $C,\alpha,\rho_0 > 0$, let $\rho = \rho_0$, and let $B = B(x,\rho)$. Since $(q_\Index)_{\Index\in\N}$ is dense in $\R$, there exist arbitrarily large $\Index$ such that $q_\Index\in B(x,\rho/2)$. For such an $\Index$, let $\beta > 0$ be chosen so that $\beta\rho = 1/b_\Index$, and assume that $\Index$ is large enough such that $\beta \leq 1/2$. Then we have
\[
\frac{\mu\big(B(q_\Index,\beta\rho)\cap B\big)}{\beta^\alpha \mu(B)}
= \frac{\mu\big(B(q_\Index,1/b_\Index)\big)}{(b_\Index\rho)^{-\alpha} \mu(B)} \geq \frac{2 a_\Index b_\Index^\alpha}{\rho^{-\alpha} \mu(B)} \tendsto{\Index\to\infty} \infty
\]
and thus \eqref{measureclass} holds for arbitrarily large $\Index$.
\end{proof}

\begin{remark}
The $\beta$ produced in the above proof can be made to satisfy $0 < \beta \leq \rho^\gamma$ for any given $\gamma > 0$, so it also shows that quasi-decay would not be a measure class invariant if we omitted the intersection with $\borel$ in the left hand side of \eqref{QDwithE}.
\end{remark}

\bibliographystyle{amsplain}

\bibliography{bibliography}

\end{document}